\documentclass[11pt, a4paper]{amsart}
\usepackage[english]{babel}
\usepackage[margin = 18mm]{geometry}

\usepackage[utf8]{inputenc}
\usepackage[T1]{fontenc}

\usepackage{amsaddr, amsmath, amsfonts, amssymb, mathtools, mathrsfs, stmaryrd}

\usepackage{tikz}
\usepackage{tikz-cd}
\usetikzlibrary{arrows,calc,shapes}
\usetikzlibrary{positioning}

\usepackage{ enumitem, float, booktabs, nicefrac, microtype, color, csquotes, mdframed, marginnote, lipsum, url, enumitem}

\usepackage[ backend=biber, style=numeric, sorting=nyt, giveninits=true, maxbibnames=99, doi=true, url=false, isbn=false, eprint=true ]{biblatex}

\renewbibmacro{in:}{
  \ifentrytype{article}
    {}
    {\printtext{\bibstring{in}\intitlepunct}}
}

\DeclareFieldFormat
  [article,incollection,inproceedings,book,thesis]
  {title}{\mkbibemph{#1}}

\DeclareFieldFormat[article]{title}{#1}

\DeclareFieldFormat{journaltitle}{\mkbibemph{#1}}
\DeclareFieldFormat{booktitle}{\mkbibemph{#1}}

\DeclareFieldFormat[article]{volume}{\mkbibbold{#1}}

\AtEveryBibitem{
  \clearlist{language}
  \clearfield{month}
  \clearfield{urldate}
}

\DeclareFieldFormat{doi}{
  \mkbibacro{DOI}\addcolon\space
  \href{https://doi.org/#1}{\nolinkurl{#1}}
}

\usepackage{hyperref}
\usepackage{cleveref}

\DeclareMathAlphabet{\pazocal}{OMS}{zplm}{m}{n}

\let\mathcal\pazocal

\newcommand{\K}{\mathsf{k}}

\DeclareMathOperator{\supp}{supp}

\newcommand{\remove}[1]{}

\theoremstyle{plain}
\newtheorem{theorem}{Theorem}[section]
\newtheorem{lemma}[theorem]{Lemma}
\newtheorem{corollary}[theorem]{Corollary}
\newtheorem{proposition}[theorem]{Proposition}

\theoremstyle{definition}

\newtheorem*{conclusion}{Conclusion}
\newtheorem*{openproblem}{Open Problem}
\newtheorem{definition}[theorem]{Definition}
\newtheorem{example}[theorem]{Example}
\newtheorem{notation}[theorem]{Notation}

\theoremstyle{remark}
\newtheorem{remark}[theorem]{Remark}

\newtheorem*{theorem*}{Theorem}
\newtheorem*{corollary*}{Corollary}

\begin{document}

\title{Properties of the $\mathcal V$-Monoid of Weighted Leavitt Path Algebras}

\author{Rishabh Goswami $^{a}$, Alfilgen Sebandal $^{b,c,d}$}
\email{rishabhgoswami.math@gmail.com, a.sebandal@rctpjagna.com}
\address{$^{a}$Department of Mathematics, North-Eastern Hill University, Shillong, India\\
$^{b}$Research Center for Theoretical Physics, Central Visayan Institute Foundation, Philippines\\
$^{c}$Department of Mathematics, Linnaeus University, V\"axj\"o, Sweden\\
$^{d}$Department of Mathematics, Central Mindanao University, Bukidnon, Philippines}


\begin{abstract}
For a row-finite weighted graph $(E,w)$, Preusser showed that the monoid $\mathcal{V}(L_\K(E,w))$ of finitely generated projective modules over the weighted Leavitt path algebra $L_\K(E,w)$ is isomorphic to a combinatorially defined weighted graph monoid $\mathcal{M}(E,w)$. We study two structural properties of $\mathcal{M}(E,w)$: confluence and cancellativity. We introduce a reduction system on the free commutative monoid presenting $\mathcal{M}(E,w)$, obtain sufficient conditions for non-confluence by   constructing  explicit non-confluent triples, and provide a complete confluence characterization for certain classes of weighted graphs. Turning to cancellativity, we work within Preusser's class of weighted graphs satisfying Condition (LPA), for which $L_\K(E,w)$ is isomorphic to an unweighted Leavitt path algebra $L_\K(F)$ via a two-step construction. We introduce an auxiliary graph associated to the intermediate step of this construction and use it to give a graph-theoretic characterization of when $\mathcal{M}(E,w)$ is cancellative. Finally, under Condition (LPA), we show that Preusser's construction upgrades to a graded isomorphism $L_\K(E,w) \cong_{\operatorname{gr}} L_\K(F)$ with respect to the standard $\mathbb{Z}^{\lambda(E,w)}$-grading of weighted Leavitt path algebras, yielding $\mathcal{V}^{\operatorname{gr}}(L_\K(E,w)) \cong \mathcal{V}^{\operatorname{gr}}(L_\K(F))$ as $\mathbb{Z}^{\lambda(E,w)}$-monoids. 
\end{abstract}

\maketitle

\section{Introduction}

 A unital ring $R$ has the \emph{Invariant Basis Number (IBN) Property}, if $R^m\cong R^n$ as $R$-modules implies $m=n$. The classical Leavitt algebras arose from Leavitt's study of rings without the IBN Property: for a field $\K$ and for positive integers $m<n$, the \emph{Leavitt $\K$-algebra} $L_\K(m,n)$ is universal with respect to an isomorphism $L_\K(m,n)^m\cong L_\K(m,n)^n$. Associated to a direct graph, \emph{Leavitt path algebras} were introduced by Abrams and Aranda Pino in 2005 \cite{AbramsArandaPino2005}, and independently by Ara, Moreno and Pardo in 2007 \cite{AraMorenoPardo07}. They contain the Leavitt algebra $L_\K(1,n)$ as a fundamental example. Associated to a weighted graph, \emph{weighted Leavitt path algebras} were introduced by Hazrat in 2013 \cite{Hazrat2013_The_Graded_Structure_of_LPAs} which includes all Leavitt algebras $L_\K(m,n)$, providing motivation for studying weighted graphs and the algebraic invariants associated with them.

For a row-finite weighted graph $(E,w)$, Preusser \cite{PreusserVMonoid} introduced a commutative monoid $\mathcal{M}(E,w)$ and proved that
$$
\mathcal{V}\bigl(L_\K(E,w)\bigr)\cong\mathcal{M}(E,w),
$$
where $\mathcal{V}(L_\K(E,w))$ denotes the monoid of isomorphism classes of finitely generated projective modules over $L_\K(E,w)$; see \cite[Theorem~14]{PreusserVMonoid}. Thus, $\mathcal{M}(E,w)$, called the \emph{weighted graph monoid} associated to $(E,w)$, provides a combinatorial description of the $\mathcal{V}$-monoid of the weighted Leavitt path algebra. In contrast with the graph monoid of an unweighted graph, the presentation of $\mathcal{M}(E,w)$ may involve generators $q_i^v$ corresponding to the different weights emitted by a vertex. These generators are precisely what distinguish the general weighted setting from the vertex-weighted one.

When $(E,w)$ is vertex-weighted, every regular vertex emits edges of a single weight and the generators $q_i^v$ do not occur. The defining relations then take the simpler form
$$
w(v)v=\sum_{e\in s^{-1}(v)}r(e).
$$
Confluence of the corresponding reduction system of the \emph{reduced weighted graph monoid} was established in \cite{AbramsHazrat23}.  In 2026, Damalerio, Hazrat and Nam \cite{dhn2026} used this confluence property to prove that the graph monoid of an acyclic vertex-weighted graph is cancellative and that the \emph{talented monoid} associated with every vertex-weighted graph is cancellative; see \cite[Theorem~2.6 and Proposition~2.20]{dhn2026}.

The situation is substantially different when the vertex-weighted assumption is removed. If a vertex emits edges of several distinct weights, the presentation of $\mathcal{M}(E,w)$ involves the  generators $q_i^v$, and several different reductions may be applicable to the same element; see Definition \ref{weighted graph monoid}. The resulting descendants need not admit a common reduction. Consequently, the confluence argument available in the vertex-weighted case does not extend formally to arbitrary weighted graphs.

Moreover, the failure of confluence is not determined solely by whether a graph is multi-weighted. As our examples show, certain multi-weighted configurations are non-confluent, while interactions with the surrounding graph, such as feedback paths, can restore confluence. Thus one should not expect a direct analogue of the vertex-weighted confluence theorem in the general setting. This motivates a systematic investigation of the reduction system associated with $\mathcal{M}(E,w)$.

We carry out this investigation in Section~\ref{sec:weighted-graph-monoid}. We introduce a reduction relation on the free commutative monoid associated with the generators of $\mathcal{M}(E,w)$ and study its confluence. To describe failures of confluence explicitly, we introduce \emph{non-confluent triples} (Definition~\ref{def:nonconfluenttriples}) and construct such triples for several classes of weighted graphs, obtaining sufficient conditions for non-confluence (Lemma \ref{lemma:nonconfluent1-revision}, Proposition~\ref{prop:non-confluent2}). In particular, for \emph{weighted star graphs}, we prove that $\mathcal{M}(E,w)$ is confluent if and only if $(E,w)$ is vertex-weighted (Theorem~\ref{theo:nonclulent_n-stargraph}). In the pursuit of obtaining a general characterization, we obtain a class of weighted graphs which appears to suggest that a general criterion for confluence requires more elaborate and case-by-case considerations (Proposition~\ref{prop:confluent_iff_beta-alpha_geq1}). With this, we conclude the section with an open problem to characterize confluence for general weighted graphs.

The second problem considered in this paper is cancellativity which we address in Section \ref{subsec:Cancellativity}. For an unweighted graph $E$, Ara, Hazrat, Li and Sims \cite{AraHazratLiSims2018} proved that the graph monoid $\mathcal{M}(E)$ is cancellative if and only if no cycle in $E$ has an exit; see \cite[Lemma~5.5]{AraHazratLiSims2018}. Their argument makes essential use of the structural tools available for unweighted graphs. In particular, for arbitrary graphs one may pass to a row-finite graph by Drinen--Tomforde desingularisation \cite{Drinen_Tomforde:2005TheC*algebrasofarbitrarygraphs}, preserving the relevant Leavitt path algebra up to Morita equivalence and hence preserving the associated graph monoid.

There is no corresponding general reduction available for weighted graphs. In particular, the weighted Leavitt path algebra of an arbitrary weighted graph need not be isomorphic to an unweighted Leavitt path algebra. This prevents one from transferring a cycle criterion for row-finite weighted case directly to the general weighted setting. This obstruction is especially relevant here because the preceding confluence analysis already shows that the ordinary rewriting arguments do not extend uniformly to arbitrary weighted graph monoids.

There is, however, a distinguished class for which such a reduction is possible. Preusser \cite{Preusser2019WeightedLP} proved that for a row-finite weighted graph $(E,w)$, the weighted Leavitt path algebra $L_\K(E,w)$ is isomorphic to an unweighted Leavitt path algebra precisely when $(E,w)$ satisfies \emph{Condition (LPA)}. Thus, the restriction to Condition (LPA) in our study of cancellativity is not merely technical: it is exactly the setting in which the weighted problem can be transported to the ordinary graph case.
More precisely, if $(E,w)$ satisfies Condition (LPA), Preusser's construction proceeds in two steps:
$$
(E,w)~\xlongrightarrow{\textnormal{Step}~\mathcal{P}1}~(\tilde{E}, \tilde{w})\xlongrightarrow{\textnormal{Step}~\mathcal{P}2}~F,
$$
where $(\tilde{E}, \tilde{w})$ is an intermediate weighted graph and $F$ is an unweighted graph such that $L_\K(E,w)\cong L_\K((\tilde{E}, \tilde{w}))\cong  L_\K(F)$ as $\K$-algebras. We then utilize the immediate consequence $\mathcal{M}(E,w) \cong \mathcal{V}\bigl(L_\K(E,w)\bigr) \cong \mathcal{V}\bigl(L_\K(F)\bigr) \cong \mathcal{M}(F) $, where the cancellativity criterion for unweighted graphs can therefore be applied to $F$. Our problem is then to express the relevant cycle structure of $F$ directly in terms of the weighted data occurring before the second step of Preusser's construction.

For this purpose, we associate to the intermediate weighted graph $(\tilde{E}, \tilde{w})$ the \emph{auxiliary graph} $\tilde{E}_{\mathrm{aux}}$ (Definition~\ref{def:auxiliarygraph}), obtained by reversing the weighted edges. Applying Step $\mathcal{P}1$ of Preusser's construction for general \emph{sink-separated weighted graphs} (Definition~\ref{def:sink-separatedweightedgraph}) we characterize when the obtained unweighted graph $H$ has no cycle with an exit based on the auxiliary graph (Proposition~\ref{prop:characterization_CE1_CE2}), and then characterize when $H$ is acyclic (Corollary~\ref{cor:cycle-exit-2}). Utilizing this to Preusser's entire two-step construction, we provide a characterization to when the weighted graph monoid of a weighted graph satisfying Condition (LPA) is cancellative (Theorem~\ref{theo:M(E,w)CancellativeCharacterization}); and when the weighted graph is acyclic, the cancellativity characterization is reduced relative to the auxiliary graph (Corollary~\ref{cor:acyclic(LPA)_Cancellative}). 

The final part of the paper (Section \ref{subsec:Vgr-onoid}) considers the graded refinement of Preusser's construction. Under Condition (LPA), we show the isomorphism $L_\K(E,w) \cong L_\K(F)$ can be upgraded to a graded isomorphism with respect to the \emph{standard $\mathbb{Z}^{\lambda(E,w)}$-grading} (Theorem~\ref{theo:gradedisomorphism_of_LPA_graphs}), yielding $\mathcal{V}^{\operatorname{gr}}(L_\K(E,w)) \cong \mathcal{V}^{\operatorname{gr}}(L_\K(F))$ as $\mathbb{Z}^{\lambda(E,w)}$-monoids (Corollary~\ref{cor:Vgr(L(E,w))=Vgr(L(F))}).

Having described the paper's contributions, we summarize its organization as follows. Section~\ref{sec:preliminaries} recalls the necessary background. Section~\ref{sec:weighted-graph-monoid} studies confluence of $\mathcal{M}(E,w)$;  Section~\ref{sec:WLPAof(LPA)weightedgraphs} considers weighted graphs satisfying the Condition $(\mathrm{LPA})$: establishing a cancellativity criterion in Section~\ref{subsec:Cancellativity}, and concluding with  a graded refinement of Preusser's construction in Section~\ref{subsec:Vgr-onoid}.

\section{Preliminaries}\label{sec:preliminaries}
\subsection{(Unweighted) Graphs}
A \emph{directed graph} is a tuple $
E=(E^0,E^1,s,r)$,
where $E^0$ and $E^1$ are disjoint sets and $s,r:E^1\to E^0$ are maps. The elements of $E^0$ and $E^1$ are called \emph{vertices} and \emph{edges}, respectively. For $e\in E^1$, $s(e)$ is called the \emph{source} of $e$ and we say $s(e)$ \emph{emits} $e$, and $r(e)$ is called the \emph{range} of $e$. A vertex $v$ is called a \emph{source} (respectively, \emph{sink}) if $r^{-1}(v)=\varnothing$ (respectively, $s^{-1}(v)=\varnothing$). If $0<|s^{-1}(v)|<\infty$, then we say $v$ is \emph{regular}. We denote by  $\textnormal{Sink}(E)$ and $\textnormal{Reg}(E)$  the sets of sinks and regular vertices in $E$, respectively.

A (finite) \emph{path}  is a sequence of edges $
p=e_1e_2\cdots e_n$
with $r(e_i)=s(e_{i+1})$ for each $1\leq i\leq n-1$. Here, $n$ is the \emph{length} of $p$ and is denoted by $|p|$. The vertices are regarded as paths of length $0$. The source of $p$ is $s(p)=s(e_1)$ and its range is $r(p)=r(e_n)$. We set $s(v)=r(v)=v$ for every vertex $v$. The set of paths in $E$ is denoted by $\textnormal{Path}(E)$. If $p=e_1\cdots e_n$ has length $n>0$ with $s(e_1)=r(e_n)$ and $s(e_i)\neq s(e_j)$ for each $i\neq j$, then $p$ is called a \emph{cycle based at $s(p)$}. An edge $f$ for which $s(f)=s(e_i)$ for some $i=1,2,\dots, n$ and $f\neq e_j$ for all $j=1,2,\dots,n$ is called an \emph{exit} of the cycle $p$. Finally, we say $E$ is \emph{acyclic} if there are no cycles in $E$.

For $u,v\in E^0$, write $u\geq v$ if there exists a path $p\in\textnormal{Path}(E)$ such that $s(p)=u$ and $r(p)=v$. For a vertex $v\in E^0$, the \emph{tree of $u$} is the set $T(u):=\{w\in E^0~|~u\geq w\}$. For $X\subseteq E^0$, define the \emph{tree of $X$} to be the set $T(X):= \bigcup_{v\in X} T(v)$. Two edges $e,f\in E^1$ are said to be \emph{in line} if $e=f$, $r(e)\geq s(f)$ or $r(f)\geq s(e)$.

The graph $E$ is said to be \emph{row-finite} if $|s^{-1}(v)|<\infty$ for every $v\in E^0$, and \emph{finite} if both $E^0$ and $E^1$ are finite sets. Throughout, by a graph we mean a directed graph.

\subsection{Graph Algebras}
Using finite paths as building blocks, we obtain the following definition of the path algebra of a graph. We refer to \cite{abramsaramolinabook,Hazrat2013_The_Graded_Structure_of_LPAs,preusser2021weighted} for a detailed account of these algebras and their properties.  

\begin{definition}[Path algebra]

For a graph $E$ and a ring $R$ with identity, the \emph{path algebra} of $E$, denoted by $P_R(E)$, is the $R$-algebra freely generated by the sets $E^0$ and $E^1$, with coefficients in $R$, subject to the relations:
\begin{enumerate}[leftmargin=2cm]
    \item [(V)] $v_iv_j=\delta_{i,j}v_i$,   for every $v_i,v_j\in E^0,$
    \item [(E)] $s(e)e=e=er(e)  $,  for every $ e\in E^1.
$
\end{enumerate}
\end{definition}

For a field $\K$, the $\K$-vector space $P_R(E)$ has basis $\textnormal{Path}(E)$ with multiplication given by concatenation of paths (which is encoded by relations \textnormal{(V)} and \textnormal{(E)}) defining the ring structure.
Throughout, we work with algebras over a field $\K$. To define Leavitt path algebras, we first define the double graph of a directed graph. 

\begin{definition}
    Let $E$ be a graph. The \emph{double graph} of $E$ is
$
\widehat{E}=(\widehat{E}^0,\widehat{E}^1,\widehat{s},\widehat{r}),
$
where
$\widehat{E}^0=E^0$,
$\widehat{E}^1=E^1\cup(E^1)^*$, $
(E^1)^*=\{e^*:e\in E^1\}$, and the maps $\widehat{s}, \widehat{r}: \widehat{E}^1\longrightarrow \widehat{E}^0$ are defined by 
\[
\widehat{s}(e)=s(e),\qquad
\widehat{r}(e)=r(e),\qquad
\widehat{s}(e^*)=r(e),\qquad \textnormal{and}\qquad 
\widehat{r}(e^*)=s(e)
\]
for every $e\in E^1$.
In other words, the elements of $(E^1)^*$ are edges with orientation opposite to that of their associated edges and are called \emph{ghost edges}.
\end{definition}
\vspace{-0.8cm}

\begin{center}
    
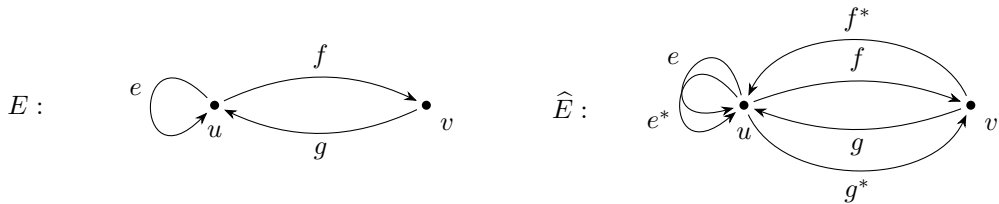
\begin{figure}[ht]
\centering
\begin{tikzpicture}[
  >={Stealth},
  vertex/.style={circle,fill=black,inner sep=1.2pt},
  every node/.style={font=\small},
  node distance=1cm
]

\node at (-5.5,0) {$E:$};
\node[vertex,label={[label distance=2pt]below:$u$}] (u1) at (-3,0) {};
\node[vertex,label=below right:$v$] (v1) at (-0.2,0) {};

\draw[->,shorten >=2pt,shorten <=2pt]
  (u1) to[bend left=25] node[above] {$f$} (v1);

\draw[->,shorten >=2pt,shorten <=2pt]
  (v1) to[bend left=25] node[below] {$g$} (u1);

\node at (1.7,0) {$\widehat{E}:$};
\node[vertex,label={[label distance=2.5pt]below:$u$}] (u2) at (4.0,0) {};
\node[vertex,label=below right:$v$] (v2) at (7.0,0) {};

\draw[->,shorten >=2pt,shorten <=2pt]
  (v2) to[bend left=-65] node[above] {$f^*$} (u2);

\draw[->,shorten >=2pt,shorten <=2pt]
  (u2) to[bend left=20] node[above] {$f$} (v2);

\draw[->,shorten >=2pt,shorten <=2pt]
  (u2) to[bend left=-65] node[below] {$g^*$} (v2);

\draw[->,shorten >=2pt,shorten <=2pt]
  (v2) to[bend left=20] node[below] {$g$} (u2);

\draw[->,shorten >=2pt,shorten <=2pt]
  (u1) edge[
    loop right,
    out=135,
    in=225,
    looseness=40
  ] node[above left] {$e$} (u1);

\draw[->,shorten >=2pt,shorten <=2pt]
  (u2) edge[
    loop right,
    out=105,
    in=195,
    looseness=40
  ] node[above left] {$e$} (u2);

\draw[->,shorten >=2pt,shorten <=2pt]
  (u2) edge[
    loop right,
    out=130,
    in=220,
    looseness=40
  ] node[below left] {$e^*$} (u2);
\end{tikzpicture}
\caption{A directed graph $E$ and its double graph $\widehat{E}$.}
\label{fig:doublegraph}
\end{figure}
\end{center}

\begin{definition}[Leavitt path algebra of a directed graph]
    
For a graph $E$ and a field $\K$, the \emph{Leavitt path algebra} of $E$, denoted by $L_\K(E)$, is the path algebra over the double graph $\widehat{E}$, subject to the \emph{Cuntz-Krieger relations}:

\begin{enumerate}[leftmargin=2cm]
    \item [(CK1)] $e^*f=\delta_{ef}r(e)$,
    for every $e,f\in E^1,$
    \item [(CK2)] $ \displaystyle\sum_{e\in s^{-1}(v)}ee^*=v$, for every $v\in\textnormal{Reg}(E)$.
\end{enumerate}
\end{definition}

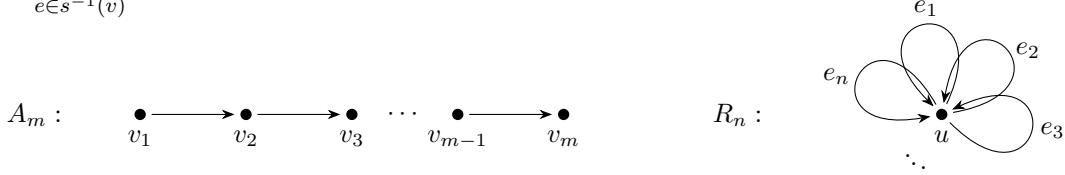
\vspace{-1cm}
\begin{center}
    
{\centering
\begin{tikzpicture}[
  >=Stealth,
  vertex/.style={circle,fill=black,inner sep=1.5pt},
  every node/.style={font=\small}
]

\node at (-5.7,0) {$A_m:$};
\node[vertex,label=below:$v_1$] (a1) at (-4.3,0) {};
\node[vertex,label=below:$v_2$] (a2) at (-2.9,0) {};
\node[vertex,label=below:$v_3$] (a3) at (-1.5,0) {};
\node at (-0.8,0) {$\cdots$};
\node[vertex,label=below:$v_{m-1}$] (anminusone) at (-0.1,0) {};
\node[vertex,label=below:$v_m$] (an) at (1.3,0) {};

\draw[->,shorten >=2pt,shorten <=2pt] (a1) -- (a2);
\draw[->,shorten >=2pt,shorten <=2pt] (a2) -- (a3);
\draw[->,shorten >=2pt,shorten <=2pt] (anminusone) -- (an);

\node at (3.6,0) {$R_n:$};
\node[vertex,label=below:$u$] (r) at (6.3,0) {};

\draw[->,shorten >=2pt,shorten <=2pt]
  (r) to[loop left, 
  out=120, in=190, looseness=50]
  node[left] {$e_n$} (r);

\draw[->,shorten >=2pt,shorten <=2pt]
  (r) to[loop above, 
  out=65, in=135, looseness=50]
  node[above] {$e_1$} (r);

\draw[->,shorten >=2pt,shorten <=2pt]
  (r) to[loop right, 
  out=10, in=80, looseness=50]
  node[right] {$e_2$} (r);

\draw[->,shorten >=2pt,shorten <=2pt]
  (r) to[loop right, 
  out=-45, in=25, looseness=50]
  node[right] {$e_3$} (r);

\node at (5.85,-0.55) {$.$};
\node at (5.93,-0.65) {$.$};
\node at (6.06,-0.7) {$.$};

\end{tikzpicture}}
\end{center}

\vspace{-0.5cm}

\begin{example} Consider the graph $A_m$ and the $n$-\emph{petal rose graph} $R_n$, $m,n\geq 1$. 
\begin{enumerate}
    \item $L_\K(A_m) \cong M_m(\K)$, the algebra of $m\times m$ matrices over $\K$.  
    \item For $n=1$, $L_\K(R_1)\cong \K[x,x^{-1}]$, the Laurent polynomial algebra. For $n>1$,  $L_\K(R_n)\cong L_\K(1,n)$, the Leavitt algebra of module type $(1,n)$. 
\end{enumerate}\end{example}
To define an algebra that recovers all Leavitt algebras $L_\K(m,n)$ of any module type $(m,n)$ with $m<n$, Hazrat~\cite{Hazrat2013_The_Graded_Structure_of_LPAs} utilized the concept of \emph{weights} on a  graph. Throughout, $\mathbb{N}=\{0,1,2,\dots\}$ and $\mathbb{N}^*=\mathbb{N}\setminus\{0\}$.

\subsection{Weighted Graphs and their Leavitt-type algebra}
A \emph{weighted graph} is a pair $(E,w)$, where $E$ is a graph and $
w:E^1\rightarrow \mathbb{N}^*$
is a map. For $e\in E^1$, $w(e)$ is called the \emph{weight} of $e$. Edges $e$ for which $w(e)>1$ (respectively, $w(e)=1$) are called \emph{weighted} (respectively, \emph{unweighted}) edges. In a weighted graph $(E,w)$, the arrow representing an edge $e$ with weight $w(e)$ shall be labelled by $``e, w(e)"$ (see Figure \ref{fig:weighted-associated-graph}). A weighted graph $(E,w)$ is said to be:

\begin{enumerate}
    \item [(a)] \emph{unweighted} if all edges in $E$ are unweighted; 
    \item [(b)] \emph{vertex-weighted} if for every $v\in\textnormal{Reg}(E)$,
$
w(e)=w(f)$ for every $ e,f\in s^{-1}(v)$; 
\item [(c)] \emph{multi-weighted} if it is not vertex-weighted, that is, there exists $u\in E^0$ such that $w(e)\neq w(f)$ for some $e,f\in s^{-1}(u)$; 

\item [(d)] \emph{connected} if for every $u,v\in E^0$, one has $u\geq v$ in the double graph $\widehat{E}$; 

\item [(e)] \emph{row-finite} (respectively, \emph{finite}) if the underlying graph $E$ is row-finite (respectively, finite); and 

\item [(f)] \emph{acyclic} if $E$ is acyclic.

\end{enumerate}

\medskip
\noindent\textbf{Assumption.}
Throughout this paper, all weighted graphs are assumed to be nonempty, row-finite, and connected, as in \cite{preusser2021weighted}.
\medskip

\begin{definition}
Let $(E,w)$ be a weighted graph. The (\emph{unweighted})
\emph{directed graph associated to} $(E,w)$ is $E^w=((E^w)^0,(E^w)^1,s^w,r^w)$,
where
\[
(E^w)^0=E^0,~~
(E^w)^1=\{e_1,\ldots,e_{w(e)}:e\in E^1\},
\text{~~and~~}s^w(e_i)=s(e),
~~
r^w(e_i)=r(e)
\]
for every $e\in E^1$. 
Since $(E,w)$ is row-finite, $E_w$ is also row-finite.   
\end{definition}

\begin{example}
    
Consider the weighted graph $(E,w)$ below, with $w(e)=2$, $w(f)=3$, and $w(g)=1$. The associated directed graph $E^w$ is obtained by replacing each edge $h\in E^1$ with $w(h)$ copies.
\end{example}
\vspace{-0.5cm}

\begin{figure}[ht]
\centering
\begin{tikzpicture}[
  >={Stealth},
  vertex/.style={circle,fill=black,inner sep=1.2pt},
  every node/.style={font=\small},
  node distance=1cm
]

\node at (-5.5,0) {$(E,w):$};
\node[vertex,label=below:$u$] (u1) at (-3,0) {};
\node[vertex,label=below:$v$] (v1) at (-0.2,0) {};

\draw[->,shorten >=2pt,shorten <=2pt]
  (u1) to[bend left=25] node[above] {$f,3$} (v1);

\draw[->,shorten >=2pt,shorten <=2pt]
  (v1) to[bend left=25] node[below] {$g,1$} (u1);

\node at (1.7,0) {$E^w:$};
\node[vertex,label=below:$u$] (u2) at (4.0,0) {};
\node[vertex,label=below:$v$] (v2) at (7.0,0) {};

\draw[->,shorten >=2pt,shorten <=2pt]
  (u2) to[bend left=70] node[above] {$f_1$} (v2);

\draw[->,shorten >=2pt,shorten <=2pt]
  (u2) to[bend left=24] node[above] {$f_2$} (v2);

\draw[->,shorten >=2pt,shorten <=2pt]
  (u2) to[bend left=-10] node[above] {$f_3$} (v2);

\draw[->,shorten >=2pt,shorten <=2pt]
  (v2) to[bend left=40] node[below] {$g_1$} (u2);

\draw[->,shorten >=2pt,shorten <=2pt]
  (u1) edge[
    loop right,
    out=135,
    in=225,
    looseness=40
  ] node[above left] {$e,2$} (u1);

\draw[->,shorten >=2pt,shorten <=2pt]
  (u2) edge[
    loop right,
    out=105,
    in=195,
    looseness=40
  ] node[above left] {$e_1$} (u2);

\draw[->,shorten >=2pt,shorten <=2pt]
  (u2) edge[
    loop right,
    out=130,
    in=220,
    looseness=40
  ] node[below left] {$e_2$} (u2);

\end{tikzpicture}
\caption{Weighted graph $(E,w)$ and its associated directed graph $E^w$.}
\label{fig:weighted-associated-graph}
\end{figure}
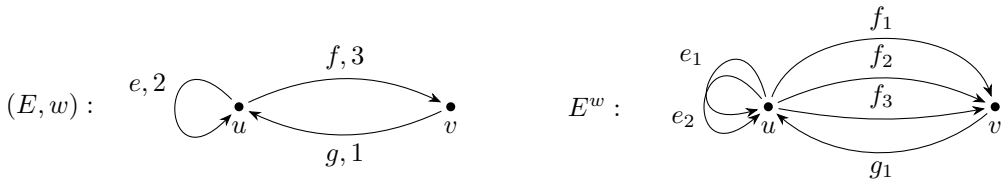

\begin{remark}
    
We note that in~\cite{Preusser2019WeightedLP}, the double graph of a graph $E$ is denoted by $E_d$, while the graph associated to a weighted graph $(E,w)$ is denoted by $\widehat{E}$. Since the notation $\widehat{E}$ is commonly used in the Leavitt path algebra literature for the double graph, we use the notation $E^w$ for the associated directed graph here instead and keep the notation $\widehat{E}$ for the double graph of $E$.
\end{remark}

Let $(E,w)$ be a weighted graph. For every $v\in\textnormal{Reg}(E)$, set $
w(v):=\max\{w(e):e\in s^{-1}(v)\}<\infty$, and for a sink $s$, set $w(s)=0$. 

\begin{definition}(Weighted Leavitt path algebra)
    Let $(E,w)$ be a weighted graph and let $\K$ be a field. The \emph{weighted Leavitt path algebra} of $(E,w)$, denoted by $L_\K(E,w)$, is the $\K$-algebra
\[
L_\K(E,w)=P_\K(\widehat{E^w})/I,
\]
where $I$ is the ideal of the path algebra $P_\K(\widehat{E^w})$ generated by the relations

\begin{enumerate} [leftmargin=2cm]
    \item [(wCK1)] $\displaystyle\sum_{1\leq i\leq w(v)} e_i^*f_i = \delta_{ef}r(e)$, ~ for every $v\in\textnormal{Reg}(E)$ and $e,f\in s^{-1}(v)$; 

    \item [(wCK2)] $\displaystyle\sum_{e\in s^{-1}(v)} e_i e_j^* = \delta_{ij}v$, ~ for every $v\in\textnormal{Reg}(E)$ and $1\leq i,j\leq w(v)$.
\end{enumerate}

Whenever $i>w(e)$, we set $e_i=e_i^*=0$ in relations (wCK1) and (wCK2). If all edges have weight $1$, then one recovers the Leavitt path algebra $L_\K(E)$; in this case, relations (wCK1) and (wCK2) are precisely the Cuntz--Krieger relations \textnormal{(CK1)} and \textnormal{(CK2)}, respectively.
\end{definition}

\begin{example}
    
Consider the graph $R_n$ consisting of one vertex and $n$ loops $e_i$. Assigning weight $w(e_i)=m< n$ for each $i$, we have $L_\K(R_n,w)\cong L_\K(m,n)
$, the Leavitt algebra of module type $(m,n)$. 

\end{example}

\medskip

\section{Confluence of \texorpdfstring{$\mathcal{V}$-monoid}{V-monoid} of Weighted Leavitt path algebras}\label{sec:weighted-graph-monoid}

For a row-finite weighted graph $(E,w)$, Preusser introduced the commutative monoid $\mathcal{M}(E,w)$ and proved that $$ \mathcal{V}\bigl(L_\K(E,w)\bigr)\cong \mathcal{M}(E,w); $$ see \cite[Theorem~14]{PreusserVMonoid}. Thus, the structure of $\mathcal{M}(E,w)$ reflects the finitely generated projective modules over $L_\K(E,w)$. In this section, we recall the presentation of $\mathcal{M}(E,w)$ and study the reduction system arising from its defining relations. Our main aim is to investigate confluence beyond the vertex-weighted setting obtained in \cite{dhn2026}. We construct explicit non-confluent triples, obtain sufficient conditions for non-confluence, and completely characterize confluence for some classes of weighted graphs.

\begin{notation}
   Let $(E,w)$ be a row-finite weighted graph. For any vertex $v\in E^0$, set $w(s^{-1}(v)):= \{ w_1(v), \dots, w_{k_v}(v) \}$, where $k_v\geq 0$ and $w_1(v)<w_2(v)<\dots < w_{k_v}(v)$. That is, $k_v$ is the number of different weights of the edges in $s^{-1}(v)$. Thus,  $k_v=1$ for every $v\in \textnormal{Reg}(E)$ if and only if $(E,w)$ is vertex-weighted.  
\end{notation}

A monoid can be associated to a weighted graph, as described in \cite{PreusserVMonoid}. In this paper, as in \cite{dhn2026}, we refer to this monoid as the \emph{weighted graph monoid}.

\begin{definition}[{\cite[Definition 11]{PreusserVMonoid}}]\label{weighted graph monoid}
       Let $(E,w)$ be a row-finite weighted graph. The \emph{weighted graph monoid} associated with $(E,w)$, denoted by  $\mathcal{M}(E,w)$, is the free commutative monoid presented by the generating set  $ \mathcal{E}=\{ v, q_1^v,\dots, q^v_{k_v-1}~\vert~ v\in E^0\}$ with the relations
    \[ 
  q_{i-1}^v+(w_i(v)-w_{i-1}(v))v=q_i^v+\sum_{\substack{e\in s^{-1}(v),\\w(e)=w_i(v)}} r(e)
    \]
for each $v\in E^0$ and $1\leq i\leq k_v$, where we set $q_0^v=q_{k_v}^v=w_0(v)=0$.  
\end{definition}

Notice that since in this paper $(E,w)$ is assumed to be row-finite, $k_v\in \mathbb{N}$ for every $v\in E^0$.

\begin{example}\label{ex:weighted_graph_monoid}
Let $(E,w)$ be a row-finite vertex-weighted graph. Then $w(v)=w_1(v)$ and $k_v=1$ for all $v\in \textnormal{Reg}(E)$. Thus $\mathcal{M}(E,w)$ is generated by the set $\{  v~|~v\in E^0\}$ subject to the relation
\[
w(v)v=\sum_{e\in s^{-1}(v)} r(e)
\] for every $v\in E^0$. Notice that since $w(s)=0$ for every sink $s$, there are no relations for sinks. Thus, $\mathcal{M}(E,w)=M_{(E,w)}$, \emph{graph monoid} of $(E,w)$ in \cite[Definition 5.5]{AbramsHazrat23}. Furthermore, \[\mathcal{M}(E,w)/\langle s=0~|~s \text{~a sink}\rangle\] is the \emph{reduced graph monoid} $M(E,w)$ in \cite[Definition 2.7]{AbramsHazrat23} where the sink is assigned a weight of $1$. 
If $(E,w)$ is a row-finite unweighted graph, then $\mathcal{M}(E,w)=M_E$, first seen in \cite{AraMorenoPardo07}, defined as the free commutative generated by $\{v~|~v\in E^0\}$ subject to the relation
\[
v=\sum_{e\in s^{-1}(v)} r(e)
\] for every non-sink $v\in E^0$. 
\end{example}

\begin{remark} \label{rem:graph_monoid=reduced_graph_monoid} The graph monoid presented in \cite{AbramsHazrat23} is defined for row-finite vertex-weighted graphs which is a specific case of the monoid defined in \cite{PreusserVMonoid} for row-finite weighted graphs (which we denote here by $\mathcal{M}(E,w)$).
    The reduced graph monoid  in \cite{AbramsHazrat23} is defined for any row-finite weighted graph (possibly multi-weighted). However, we have the following which was also noted in \cite{AbramsHazrat23}. If $(E,w)$ is a sinkless vertex-weighted graph, then $M_{(E,w)}=M(E,w)$. If $(E,w)$ has sinks, then $M_{(E,w)}$ is equal to the reduced graph monoid $M(E',w')$ of some sinkless weighted graph $(E',w')$. In other words, for any vertex-weighted graph $(E,w)$, there exists a sinkless graph $(E',w')$ for which $M_{(E,w)}=M(E',w')$.
\end{remark}

\subsection{Confluence of Weighted Graph Monoid}

\begin{definition}
Let $A$ be a nonempty set and $\mapsto~\subseteq A\times A$. We call $\mapsto$ a \emph{reduction} or \emph{rewriting  on A} and the pair $(A, \mapsto)$ an \emph{abstract  reduction (or rewriting) system}. For $(a,b)\in~\mapsto$, we write $a\mapsto b$ and call $a\mapsto b$ a \emph{reduction step}, or simply a \emph{reduction}. Denote by $\mapsto^*$ the transitive reflexive closure of $\mapsto$. Hence, any reduction $a\mapsto^* b$ is a (finite) \emph{composition} of reductions $a\mapsto x_1\mapsto \dots\mapsto  x_n\mapsto b$. Such a sequence of reductions is called a \emph{chain}.

An element $x\in A$ is said to be \emph{reducible} if there exists $y\in A$ such that $y\neq x$ and $x\mapsto^*y$. In this case, we say $x$ \emph{is reduced to $y$}, and $y$ is a \emph{reduction} of $x$. If no such $y$ exists, we say $x$ \emph{cannot fire} or is \emph{irreducible}. We say $x, y\in A$ are \emph{joinable} if there exists $z\in A$ such that $x\mapsto^* z$ and $y\mapsto^* z$. 
\end{definition}

\begin{definition}
    Let $(A, \mapsto)$ be an abstract reduction system. We say $\mapsto$ is \emph{confluent} if for every $x, a, b\in A$ with $x \mapsto^* a$ and $x\mapsto^* b$, then $a$ and $b$ are joinable. That is, there exists $c\in A$ such that $a\mapsto^* c$ and $b\mapsto^* c$. In this case, we call $(A,\mapsto)$ a \emph{confluent} abstract reduction system. If no such element $c$ exists, we say $(A, \mapsto)$ is  \emph{non-confluent on  the element $x$}, or simply \emph{non-confluent}. 
\end{definition}
   
Our point of interest is to investigate the confluence of the weighted graph monoid $\mathcal{M}(E,w)$ where elements are equal when they ``flow'' to the same element. For this, we shall work on a certain abstract reduction system $(F_\mathcal{E}, \rightarrow)$ for which $\mathcal{M}(E,w)=F_\mathcal{E}/\sim$, where $\sim$ is the congruence relation generated by $\rightarrow^*$.  In \cite{AraMorenoPardo07}, the graph monoid of an unweighted graph was described by a relation on the free monoid on the set of vertices. In \cite{AbramsHazrat23}, the relation was called the  ``\textbf{r}-transform'' generalizing to the vertex-weighted graphs. We shall extend this to arbitrary weighted graphs.

For a non-empty set $S$, denote by $F_{S}$ the free commutative monoid on $S$. Let $(E,w)$ be a weighted graph and let  $\mathcal{E}= E^0\cup E^q$ where $E^q=\{ q_1^v,\dots, q^v_{k_v-1}~|~v\in E^0\}$. For each $v\in E^0$, set $w_0(v)=q_0^v=q_{k_v}^v=0$ and for each $1\leq i\leq k_v$.

For each $v\in \textnormal{Reg}(E)$ and  $1\leq i\leq k_v$ we define the binary relations $\textnormal{\textbf{r}}_{0,i}, \textnormal{\textbf{r}}_{1,i}$ on $ F_{E^0}$ as follows:  
\[
\textnormal{\textbf{r}}_{0,i}(nv)=(n-w_i(v)+w_{i-1}(v))v ~~~(n \geq w_i(v)-w_{i-1}(v)) 
\]
 and
\[
\textnormal{\textbf{r}}_{1,i}((w_i(v)-w_{i-1}(v))(v))= \sum_{\substack{e\in s^{-1}(v),\\w(e)=w_i(v)}} r(e).
\]
Now for $v\in E^0$ and $1\leq i\leq k_v$, define the \emph{$\textnormal{\textbf{r}}_i$-transform} as the binary relation $\textnormal{\textbf{r}}_i$ on  $F_{E^0}$ given by
\[
\textnormal{\textbf{r}}_i(nv)=\textnormal{\textbf{r}}_{0,i}(nv)+\textnormal{\textbf{r}}_{1,i}((w_i(v)-w_{i-1}(v))(v))  ~~~(n\geq w_i(v)-w_{i-1}(v)).\]

 For each $v\in E^0$ and  $0\leq i\leq k_v-1$, define the binary relations $\textnormal{\textbf{q}}_{0,i},\textnormal{\textbf{q}}_{1,i}$ on $F_{E^q}$  as follows: 
\[
\textnormal{\textbf{q}}_{0,i}(nq_{i}^v)= (n-1)q_{i}^v~~(n\geq 1)~
\text{~and~}
\textnormal{\textbf{q}}_{1,i}(q_{i}^v)= q_{i+1}^v.
\]
For $v\in E^0$ and $0\leq i\leq k_v-1$, define the \emph{$\textnormal{\textbf{q}}_i$-transform} as the binary relation $\textnormal{\textbf{q}}_i$ on $F_{E^q}$ given by
\[\textnormal{\textbf{q}}_i(nq_{i}^v)= \textnormal{\textbf{q}}_{0,i}(nq_{i}^v)+\textnormal{\textbf{q}}_{1,i}(q_{i}^v) ~~(n\geq 1).\]

The elements of $F_{\mathcal{E}}$ are of the form
\[x=
\sum_{v\in E^0}\sum_{j=-1}^{k_v-1}\alpha_j^vq^v_j
\]
where for each $v\in E^0$, $q_{-1}^v:=v$ 
 and $\alpha_j^v\in \mathbb{N}$ (nonzero for a finite number) for every $-1\leq j\leq k_v-1$. 
 Define the reduction $\rightarrow$ on $F_{\mathcal{E}}$ for each $u\in E^0$ and $1\leq i\leq k_u$ by 
\begin{equation}
x=\sum_{v\in E^0}\sum_{j=-1}^{k_v-1}\alpha_j^vq^v_j~ \longrightarrow~ \sum_{v\neq u}\sum_{j=-1}^{k_v-1}\alpha_j^vq^v_j + \sum_{j\neq -1,i-1}\alpha_j^uq_j^u +\textnormal{\textbf{q}}_{i-1}{(\alpha_{i-1}^uq_{i-1}^u)} + \textnormal{\textbf{r}}_i(\alpha_{-1}^uq_{-1}^u), \tag{$\mathfrak{R}$} \label{reductions}
\end{equation}
whenever $\alpha_{i-1}^u\neq 0$ and $\alpha _{-1}^u\geq  w_i(u)-w_{i-1}(u)$. In other words, the reduction $\rightarrow$ on an element $x\in F_{\mathcal{E}}$ can only occur whenever there exist $u\in E^0$ and $1\leq i\leq k_u$ such that $q_{i-1}^u$ and $(w_i(u)-w_{i-1}(u))(u)$ are summands of $x$.

\begin{notation}\label{nota:Sigma(u)}
    For a weighted graph $(E,w)$ and for every $u\in \textnormal{Reg}(E)$ and $\alpha \in w(s^{-1}(u))$, denote 
    \[
    \Sigma_\alpha(u):= \sum_{\substack{e\in s^{-1}(u)\\ w(e)=\alpha}} r(e).
    \]
\end{notation} Then explicitly, for each $u\in E^0$ and $1\leq i\leq k_u$, we obtain the following \emph{reductions at $u$}
\begin{equation*}
 q_{i-1}^u+(w_i(u)-w_{i-1}(u))u~\longrightarrow~ q_i^u+\Sigma_{w_i(u)}(u)\quad (1\leq i\leq k_u).\tag{$\mathfrak{R}i(u)$} \label{reductions_at_u}
\end{equation*}
We write $x \xrightarrow{(\mathfrak{R}_i(u))} y$ when applying the reduction $(\mathfrak{R}_i(u))$ on the element $x$. If no confusion arises, we sometimes denote the reductions on $u$ simply by $(\mathfrak{R}i)$ or $(\mathfrak{R}(u))$. 
Let $\rightarrow^*$ be the transitive and reflexive closure of $\rightarrow$, that is, whenever $a\rightarrow^* b$, either $a=b$, or there exist $a_0, a_1, \dots , a_n\in F_{\mathcal{E}}$ such that $a=a_0\rightarrow a_1 \rightarrow a_2 \rightarrow\dots \rightarrow a_n=b$. Let $\sim$ be the congruence relation generated by $\rightarrow^*$. Then $a\sim b$ whenever there exist $a=a_0,a_1,a_2, \dots , a_{n}=b\in F_\mathcal{E}$ such that $a_i\rightarrow a_{i+1}$ or $a_{i+1}\rightarrow a_i$ for each $0\leq i \leq n-1$.  Therefore, 
\[
\mathcal{M}(E,w)=F_{\mathcal{E}}/\sim. 
\]
We call $F_\mathcal{E}/\sim$ the \emph{canonical presentation} of $\mathcal{M}(E,w)$.

\begin{definition}
Let $(A,\mapsto)$ be an abstract reduction system and $\mapsto^\text{c}$ be the congruence relation generated by $\mapsto^*$. We say the presentation $A/\mapsto^\text{c}$  is \emph{confluent} if for $a,b\in A$, $a\mapsto^\text{c} b$ if and only if $a$ and $b$ are joinable, that is, there exists $c\in A$ such that $a\mapsto^* c$ and $b\mapsto^* c$. 
\end{definition}

\begin{lemma}\label{lemma:diamondlemma-like}
        Let $(A, \mapsto)$ be an abstract reduction system and $\mapsto^\text{c}$ be the congruence relation generated by $\mapsto^*$. Then $(A, \mapsto)$ is confluent if and only if $A/\mapsto^{\text{c}}$ is confluent. 
\end{lemma}

\begin{proof}
$(\Leftarrow)$ Immediate.

\noindent $(\Rightarrow)$ Suppose $(A, \mapsto)$ is confluent and $a\mapsto^\text{c} b$. Then there exists a string of length $n$, $a=a_0,a_1,a_2, \dots, a_n=b\in A$ for which $a_i\mapsto a_{i+1}$ or $a_{i+1}\mapsto a_i$. For $n=1$, the conclusion clearly holds. For $n=2$, we have the following cases: $a\mapsto a_1 \mapsto b$, $a\mapsto a_1\mapsfrom b$, $a\mapsfrom a_1 \mapsfrom b$, and $a\mapsfrom a_1 \mapsto b$. Clearly, for the first three cases, the conclusion holds. For $a\mapsfrom a_1 \mapsto b$, there exists $c\in A$ such that $a\mapsto^* c$ and $b\mapsto^* c$, due to the confluence of $(A, \mapsto)$. Suppose the conclusion holds for strings of length less than $n$. Now, $a\mapsto^\text{c}a_{n-1}$ with a string of length $n-1$. Thus, there exists $c\in A$ such that $a\mapsto^* c$ and $a_{n-1}\mapsto^* c$.
\\

\noindent Case 1. $a_{n-1}\mapsto a_n=b$. Since $(A, \mapsto)$ is confluent, there exists $c' \in A$ such that $b\mapsto^* c'$ and $c\mapsto^* c'$. Thus, $a\mapsto^* c\mapsto^* c'$ and $b\mapsto^* c'$. \\

\noindent Case 2. $a_{n-1}\mapsfrom b$. Then $b\mapsto a_{n-1}\mapsto^* c$. Thus, $a\mapsto^* c$ and $b\mapsto^* c$. \end{proof}

\begin{remark}
    From then on, for a weighted graph $(E,w)$, we shall say $\mathcal{M}(E,w)$ is confluent whenever its canonical presentation $F_{\mathcal{E}}/\sim$ is confluent and when talking about elements of the free monoid $F_\mathcal{E}$,
irreducibility is understood to be relative to the reductions (\ref{reductions}).  
Confluence in the context of (unweighted) graph monoids was first seen in \cite{AraMorenoPardo07} and the version for the reduced graph monoid for weighted graphs was seen in \cite{AbramsHazrat23}. For vertex-weighted graphs, we have the following: 
\end{remark}
\begin{lemma}[{\textnormal{\cite[Lemma 2.2]{
dhn2026}}}]\label{lem:confluence_vertex-weighted_graphs} Let $(E,w)$ be a vertex-weighted graph. Then  its associated weighted graph monoid $\mathcal{M}(E,w)$ is confluent.  
\end{lemma}

\begin{example}\label{ex:nonconfluentS2}
        Let $(E,w)$ be the weighted graph shown below with weights $\alpha, \beta$,  $\alpha <\beta$.
        
\begin{center}

\begin{tikzpicture}[>=stealth, node distance=3cm, baseline=(u.base)]
    \node (E) at (-3.2,0) {$(E,w):$};
    \node (s1) at (-2,0) {$s_1$};
    \node (u)  at (0,0) {$u$};
    \node (s2) at (2,0) {$s_2$};

    \draw[->] (u) -- node[above] {$e_1,\alpha$} (s1);
    \draw[->] (u) -- node[above] {$e_2,\beta$} (s2);
\end{tikzpicture}

\end{center}
Then $k_u=2$,   $\mathcal{E}=\{u,s_1,s_2,q\}$ where $q:=q^u_1$,  and the abstract reduction system $(F_\mathcal{E}, \rightarrow)$ have the following reductions at $u$: 
\begin{center}
    $(\mathfrak{R}1) \quad \alpha u\longrightarrow q+s_1$ ~~~and~~~ $(\mathfrak{R}2) \quad q + (\beta -\alpha)u\longrightarrow s_2$,
\end{center}
together with the other possible reductions described in (\ref{reductions}) for general elements of $F_\mathcal{E}$. 
By the division algorithm, $\beta = m\alpha +d $, where $m\geq 1$ and $0\leq d < \alpha$. \\~\\
Case 1: $m=1$. Then $\beta=\alpha +d$, that is, $\beta \in [\alpha , 2\alpha)$. Then 
\[x_1:= q+\alpha u \stackrel{(\mathfrak{R}1)}{\longrightarrow} q+q+s_1=2q+s_1=: a_1\] and 
\[x_1:= q+\alpha u = q+d u +(\alpha -d )u \stackrel{(\mathfrak{R}2)}{\longrightarrow} s_2+ (\alpha-d)u=: b_1.\] 
Case 2: $\beta=m\alpha $, $m\geq 2$. Then 
\[
x_2:= q+\beta u = q+ (\beta -\alpha )u+ \alpha u     \stackrel{(\mathfrak{R}2)}{\longrightarrow } s_2+ \alpha u \stackrel{(\mathfrak{R}1)}{\longrightarrow} s_2+q+s_1 =: a_2
\]
and 
\[
x_2:= q+\beta u = q+m\alpha \stackrel{(\mathfrak{R}1)}{\longrightarrow} q+m(q+s_1)= (m+1)q+ms_1=: b_2 
\]
Case 3: $m\alpha <\beta <(m+1)\alpha $, $m\geq 2$. Then 
\[
q+\beta u = q+(\beta -\alpha) u + \alpha u\stackrel{(\mathfrak{R}2)}{\longrightarrow} s_2+ \alpha u \stackrel{(\mathfrak{R}1)}{\longrightarrow} s_2+q+s_1
\]
and 
\[
 q+\beta u = q+m\alpha u +du\stackrel{(\mathfrak{R}1)}{\longrightarrow} q+m(q+s_1) + du= (m+1)q+ms_1+du. 
\]
Note that $d<\beta -\alpha$. Then let $x_3:= q+\beta u$, $a_3:= s_2+q+s_1$ and $b_3:=(m+1)q+ms_1+du$.

For each of the cases,  $x_i\rightarrow^* a_i$ and $x_i\rightarrow^* b_i$ with $a_i\neq b_i$ and $a_i$ and $b_i$ are irreducible $(i=1,2,3)$. Hence, $(F_\mathcal{E}, \rightarrow)$ is not confluent, and so is 
$\mathcal{M}(E,w)$.

\end{example}

We account such triples $(x, a, b)$ for which $x\rightarrow^* a$ and $x\rightarrow^* b$, where $a$ and $b$  are not joinable.

\begin{definition}\label{def:nonconfluenttriples}
Let $(A,\mapsto)$ be a non-confluent abstract reduction system. We call a triple $(x,a,b)\in A\times A\times A$ a 
\emph{non-confluent triple for $A$} if  $x\mapsto^* a$, $x\mapsto^* b$ and $a$ and $b$ are not joinable. Due to Lemma \ref{lemma:diamondlemma-like}, we also say such $(x,a,b)$ is a non-confluent triple in $A/\mapsto ^{\text{c}}$. 
\end{definition}

\begin{example}
    
    Consider the non-confluent graph monoid $\mathcal{M}(E,w)$ in Example \ref{ex:nonconfluentS2}. Then $(x_i, a_i, b_i)$, $i=1,2,3,4$, is a non-confluent triple for each case therein.   
\end{example}

\begin{remark}
   Constructing explicit non-confluent triples
   as in Example \ref{ex:nonconfluentS2}
   can in general be a lengthy and technically involved process due to the large number of cases that must be considered. In many applications, however, an appropriate estimate of such a triple suffices, as to be illustrated in the proof of Lemma  \ref{lemma:nonconfluent1-revision}.

\end{remark}

\begin{definition} \label{def:*-graph}
    A connected graph $E=(E^0, E^1, s,r)$ is called a \emph{star graph} if
    \begin{enumerate}
        \item [(i)] $E^0 = \{u\}\cup S$, where $S$ is a nonempty set;
        \item  [(ii)] for each $e\in E^1$, 
$s(e)=u$ and $r(e)=s$ for some $s\in S$.
    \end{enumerate}
 
\end{definition}

\begin{remark} In Definition \ref{def:*-graph}, 
   it follows that the vertex $u$ is the unique  source and  $S$ is the set of sinks in $E$ and if $(E,w)$ is row-finite, then $|S|=n< \infty$. In this case, we shall call $(E,w)$ an \emph{$n$-star graph with source vertex $u$} or simply an \emph{$n$-star graph}. 
\end{remark}

\begin{example}
Below are examples of $n$-star graphs with $n=2$, $n=m+2$ ($m\in \mathbb{N}^*$) and $n=1$, respectively. 
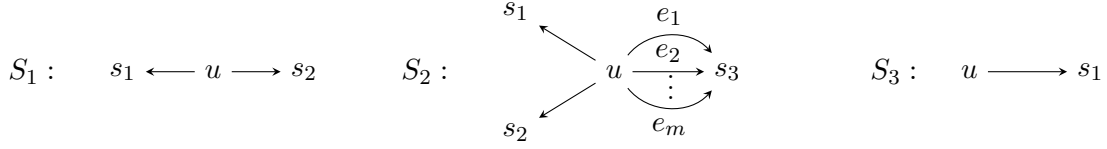
\begin{center}
\begin{tikzpicture}[>=stealth, scale=1, every node/.style={scale=1}]
    \node at (-5.2,0) {$S_1:$};
    \node (s1a) at (-4.0,0) {$s_1$};
    \node (u1)  at (-2.8,0) {$u$};
    \node (s2a) at (-1.6,0) {$s_2$};

    \draw[->] (u1) -- (s1a);
    \draw[->] (u1) -- (s2a);

    \node at (0.0,0) {$S_2:$};
    \node (s1b) at (1.2,0.8) {$s_1$};
    \node (s2b) at (1.2,-0.8) {$s_2$};
    \node (u2)  at (2.5,0) {$u$};
    \node (s3b) at (4.0,0) {$s_3$};

    \draw[->] (u2) -- (s1b);
    \draw[->] (u2) -- (s2b);

    \draw[->, bend left=50] (u2) to node[above] {$e_1$} (s3b);
    \draw[->] (u2) to node[above] {$e_2$} (s3b);
    \draw[->, bend right=50] (u2) to node[below] {$e_m$} (s3b);

    \node at (3.25,-0.1) {$\vdots$};

    \node at (6.2,0) {$S_3:$};
    \node (u3)  at (7.2,0) {$u$};
    \node (s1c) at (8.8,0) {$s_1$};

    \draw[->] (u3) -- (s1c);
\end{tikzpicture}
\end{center}
\end{example}

\begin{definition}
\label{def:c_s,mu(s),support}    
Let $(E,w)$ be a weighted graph and $F_\mathcal{E}$ be the free commutative monoid on $\mathcal{E}=\{v, q^v_1, \dots, q^v_{k_v-1}~|~ v\in E^0\}$. 
  For every $s \in \mathcal{E}$, define 
$ c_s: F_\mathcal{E}\longrightarrow \mathbb{N}$ by  \[
c_s\left( \sum_{w\in \mathcal{E}} \gamma_w w\right ) =\gamma_s. \]
For $x\in F_\mathcal{E}$, denote by $\supp(x)$ the \emph{support of $x$}, that is, $\supp(x)=\{ s\in \mathcal{E} ~|~ c_s(x)\neq 0 \}$. Let $\supp_0(x)=\supp(x)\setminus E^q=\{ s\in E^0~|~ c_s(x)\neq 0 \}$ which we shall call the \emph{vertex-support} of $x$. 

 Define $\mu:E^0\longrightarrow \mathbb{N}\cup \{\infty \}$ by 
\[
\mu(s):=
\begin{cases}
\min\{w_i(s)-w_{i-1}(s)~|~ 1\leq i\leq k_s\} & \textnormal{~if~} s\in \textnormal{Reg}(E),\\[4pt]
\infty & \textnormal{~if~} s\in \textnormal{Sink}(E),
\end{cases}
\] for every $s\in E^0$, where $w_0:=0$. 

\end{definition}  

\begin{remark}
Let $(E,w)$ be a weighted graph and $F_{\mathcal{E}}$ be the free commutative monoid on $\mathcal{E}=\{v, q_1^v,\dots, q^v_{k_v-1}~|~v\in E^0\}$.

\begin{enumerate}
    \item The reduction (\ref{reductions_at_u}) at a vertex $u$ is applicable to  $x\in F_\mathcal{E}$ if and only if $q_{i-1}^u\in \supp(x)$ and $c_u(x)\geq  w_i(u)-w_{i-1}(u)$. 
\item Suppose any of the following are satisfied for $x\in F_\mathcal{E}$:
\begin{enumerate}
    \item [(i)]
  $\mu(s)>c_s(x)$ for every $s\in \supp_0(x)$;  \item [(ii)] $x\in F_{E^0}$ and  $w_1(s)>c_s(x)$ for every $s\in \supp_0(x)$; 
\item [(iii)] $\{q_{i-1}^u, (w_{i}(u)-w_{i-1}(u))u\}\not \subseteq \supp(x)\cup \{0\}$ for every $u\in E^0$ and $1\leq i\leq k_u$.
\end{enumerate}
Then $x$ is irreducible. 
\end{enumerate}

\end{remark}

\begin{lemma}
\label{lemma:nonconfluent1-revision}  Let $(E,w)$ be a multi-weighted graph. If there exists $u\in E^0$ with $k_u>1$ and $u\not \in T(u)\setminus \{u\}$, then $\mathcal{M}(E,w)$ is non-confluent. 
\end{lemma}

\begin{proof}
    Let $(E,w)$ be a multi-weighted graph and $F_\mathcal{E}$ be the free commutative monoid on the generating set  $\mathcal{E}=\{v, q_1^v, q_2^v, \dots q_{k_v-1}~|~ v\in E^0\} $. Suppose there exists $u\in E^0$ with $k_u>1$ and $u\not \in T(u)\setminus\{u\}$. Let $w(s^{-1}(u))=\{\alpha_1, \alpha_2,\dots, \alpha_{k_u}\}$, $q_i^u:= q_i$ and $\Sigma_i:=\Sigma_{\alpha_i}(u)$.  Then in $(F_\mathcal{E}, \rightarrow)$, the reductions at $u$ are:
    \begin{equation*}
{(\mathfrak{R}i)}\quad q_{i-1}+(\alpha_i-\alpha_{i-1})u~\longrightarrow~ q_i+\Sigma_i\quad (1\leq i\leq k_u).~~~ 
\end{equation*}
By division algorithm, $\alpha_2=n\alpha_1+d$ where $n\geq 1$ and $0\leq d<\alpha_1$.\\

\noindent Case 1. $n=1$. Let $x=q_1+\alpha_1 u$. Then 
\[
x=q_1+\alpha_1u\xrightarrow{(\mathfrak{R}1)} q_1+(q_1+\Sigma_1)= 2q_1+\Sigma_1:= a. 
\]
Since $u\not \in T(u)\setminus \{ u\}$, $c_u(a)=0$.  Thus, for every $i$, reduction $(\mathfrak{R}i)$ is not applicable to $a$. This implies that for every reduction $a'$ of $a$ is of the form $2q_1+y$ with $u\not \in \supp(y)$. 
 
Since $n=1$,  $2\alpha_1-\alpha_2=2\alpha_1-(\alpha_1+d)=\alpha_1-d>0$. Hence, we can write $\alpha_1=(\alpha_2-\alpha_1)+(2\alpha_1-\alpha_2)$. Now, 
\[
x=q_1+\alpha_1u = q_1 +  (\alpha_2-\alpha_1)u+(2\alpha_1-\alpha_2)u\xrightarrow{(\mathfrak{R}2)} (q_2+\Sigma_2) + (2\alpha_1-\alpha_2)u:= b
\]
Moreover, since $n=1$ and $\alpha_2>\alpha_1$, $d\neq 0$. Since $u\not \in T(u)\setminus\{u\}$, $c_u(b)=2\alpha_1-\alpha_2=\alpha_1-d<\alpha_1$. Hence, $(\mathfrak{R}1)$ is not applicable to $b$ and since $q_1\not \in \supp(b)$, any reduction $b'$ of $b$ is of the form \[
b'= z +\sum_{j\geq 2}c_j{q_j}+fu
\]
where $c_j,f\in \mathbb{N}$, $f\leq 2\alpha_1-\alpha_2<\alpha_1$, and $z\in F_\mathcal{E}$ with $u\not \in \supp(z)$. Thus, for any reduction $b'$ of $b$, $q_1\not \in \supp(b')$. Hence, for any reduction  $a'$ of $a$ and $b'$ of $b$, $a'\neq b'$. 
\\

\noindent Case 2: $n\geq 2$. Let $x=q_1+\alpha_2 u$.  Then 
\[
x
=
[q_1+(\alpha_2-\alpha_1)u+\alpha_1 u]
\xlongrightarrow{(\mathfrak{R}2)}
(q_2+\Sigma_{2})+\alpha_1 u
\xlongrightarrow{(\mathfrak{R}1)} q_2+ 
\Sigma_2+(q_1+\Sigma_1)=q_1+q_2+\Sigma_1+\Sigma_2 := a.
\] 
Now, $u\not \in T(u)\setminus \{u\}$ implies that for any reduction $a'$ of $a$, $a'= q_1+q_2+y$ where $u,q_1,q_2\not \in \supp(y)$. \\

Now, $x=q_1+\alpha_2 u = q_1+(n\alpha_1 +d)u=q_1+n\alpha_1 u +d u$. Applying $(\mathfrak{R}1)$ $n$-times, 
\[
x
\longrightarrow^*
q_1+n(q_1+\Sigma_1)+d u=
(n+1)q_1+n\Sigma_1+d u:=b. 
\]
Since $u\not \in T(u)\setminus \{u\}$, $c_u(b)=d$. Since $\alpha_1>d$ and $\alpha_2-\alpha_1=n\alpha_1+d-\alpha_1=(n-1)\alpha_1+d>d$, it follows that $(\mathfrak{R}1)$ and $(\mathfrak{R}2)$ are not applicable to $b$. Moreover, since $u\not \in T(u)\setminus \{u\}$, for any reduction $b'$ of $b$, $b'= z+(n+1)q_1+fu$ where $q_1, q_2,u\not \in \supp(z)$ and $f\leq d$. Hence, for any reduction $a'$ of $a$ and $b'$ of $b$, $a'\neq b'$

In both cases, we obtain elements $x,a,b\in F_{\mathcal{E}}$ with $x\rightarrow^* a$,  $x\rightarrow^* b$ and for any reduction $a'$ of $a$ and $b'$ of $b$, $a'\neq b'$. Accordingly, we obtain a non-confluent triple $(x,a,b)$. Therefore, $(F_\mathcal{E}, \rightarrow)$ is non-confluent. By Lemma \ref{lemma:diamondlemma-like}, $\mathcal{M}(E,w)$ is non-confluent.
\end{proof}

\begin{theorem}
    
\label{theo:nonclulent_n-stargraph}   Let $(E,w)$ be a weighted $n$-star graph and $\mathcal{M}(E,w)$ its associated weighted graph monoid. Then $\mathcal{M}(E,w)$ is confluent if and only if $(E,w)$ is vertex-weighted.

\end{theorem}

\begin{proof}

\noindent By Lemma \ref{lem:confluence_vertex-weighted_graphs} and Lemma \ref{lemma:nonconfluent1-revision}, the theorem directly follows. For completeness, we show an explicit computation of non-confluent triples for the multi-weighted $n$-star graph. 

Let $(E,w)$ be a  multi-weighted $n$-star graph with $E^0=\{u, s_1, s_2, \dots, s_n\}$, where the $s_i$'s are sinks, and weights $\alpha_1<\alpha_2<\dots <\alpha_m$. Then $k_u=m>1$ and $\mathcal{M}(E,w)$ is generated by the set $\mathcal{E}=\{ u, s_1, \dots, s_n, q_1, \dots, q_{m-1}\}$ where $q_i:= q_i^u$.
For each $i$, let $\displaystyle\Sigma_i:=\Sigma_{\alpha_i}(u)$. Then $(F_\mathcal{E}, \rightarrow)$ have the reductions
\[
(\mathfrak{R}i)~~~~~
  q_{i-1}+(\alpha_i-\alpha_{i-1})u~\longrightarrow~ q_i+\Sigma_i ~~
    \]
for each $1\leq i\leq m$, where we set $q_0=q_{m}=\alpha_0=0$, and together with the other possible reductions described  in (\ref{reductions}) for general elements of $F_\mathcal{E}$. By division algorithm, $\alpha_2=n\alpha_1+d$, where $n\geq 1$ and $0\leq d< \alpha_1$.\\

\noindent Case 1: $n=1$. Then $\alpha_2=\alpha_1+d$ where $0< d<\alpha_1$. Let $x:= q_1+\alpha_1u$, $a=2q_1+\Sigma_1$, and $b_1=\Sigma_2+q_2+(2\alpha_1 -\alpha_2)u$. 

As shown in the proof of Lemma \ref{lemma:nonconfluent1-revision}, $a$ is irreducible, $x\rightarrow^* a$, and $x\rightarrow^*b_1$. Since $c_u(b_1)=2\alpha_1-\alpha_2=2\alpha_1-(\alpha_1+d)=\alpha_1-d<\alpha_1$, $(\mathfrak{R}1)$ is not applicable to $b_1$. For $i\neq 1, 3$, $(\mathfrak{R}i)$ is not applicable to $b_1$ since $q_{i-1}\not \in \supp(b_1)$. Since $q_2\in \supp(b_1)$, $(\mathfrak{R}3)$ may be applied to $b_1$, depending on the value of $2\alpha_1-\alpha_2$. 

~\\
\noindent Subcase 1.1: $2\alpha_1<\alpha_3$. Then  $2\alpha_1-\alpha_2< \alpha_3-\alpha_2$ and  $(\mathfrak{R}3)$ is not applicable to $b_1$ and hence, $b_1$ is irreducible.\\
\noindent Subcase 1.2: $2\alpha_1\geq \alpha_3$. Then 
\[
b_1 \xrightarrow{(\mathfrak{R}3)} \Sigma_2+(2\alpha_1 -\alpha_2-(\alpha_3-\alpha_2))u+ (q_3+\Sigma_3)= \Sigma_2+\Sigma_3+q_3+(2\alpha_1 -\alpha_3)u:=b_2 
\]
Since $c_u(b_2)=2\alpha_1 -\alpha_3< \alpha_1$, $(\mathfrak{R}1)$ is not applicable to $b_2$. For $i\neq 1, 4$, $(\mathfrak{R}i)$ is not applicable to $b_2$ since $q_{i-1}\not \in \supp(b_2)$. Since $q_3\in \supp(b_2)$, $(\mathfrak{R}4)$ may be applied to $b_2$, depending on the value of $2\alpha_1 -\alpha_3$.\\

\noindent Subcase 1.2.1. $2\alpha_1 <\alpha_4$. Then   $2\alpha_1-\alpha_3< \alpha_4-\alpha_3$ and  $(\mathfrak{R}4)$ is not applicable to $b_2$ and hence, $b_2$ is irreducible. \\
Subcase 1.2.2. $2\alpha_1 \geq \alpha_4$. Then 
\[
b_2  \xrightarrow{(\mathfrak{R}4)} 
\Sigma_2+\Sigma_3+(2\alpha_1 -\alpha_3-(\alpha_4-\alpha_3))u+ (q_4+\Sigma_4)= \sum_{i=2}^4\Sigma_i+q_4+(2\alpha_1-\alpha_4)u:= b_3. 
\]

Continuing this process, we either have  $2\alpha_1<\alpha_{s+1}$ for some $2\leq s\leq m-1$ and we obtain an irreducible  element $b_{s-1}$ of the form
\[
b_{s-1}= \sum_{i=2}^s\Sigma_i+q_{s}+(2\alpha_1-\alpha_s)u, 
\]
or $2\alpha_1\geq \alpha_m$. In this case, we have the element
\[
b_{m-1}= \sum_{i=2}^m\Sigma_i+q_{m}+(2\alpha_1-\alpha_m)u = \sum_{i=2}^m\Sigma_i+(2\alpha_1-\alpha_m)u.
\]

Since $c_u(b_{m-1})=2\alpha_1-\alpha_m< 2\alpha_1-\alpha_2<\alpha_1$, $(\mathfrak{R}1)$ is not applicable to $b_{m-1}$. Since for all $i=1,2,\dots,  m$,  $q_{i-1}\not \in \supp(b_{m-1})$, $(\mathfrak{R}i)$ is not   applicable to $b_{m-1}$. Hence, $b_{m-1}$ is irreducible. Thus, for some $2\leq z\leq m$. 
\[
b_{z-1}= \sum_{i=2}^z\Sigma_i+q_{z}+(2\alpha_1-\alpha_z)u. 
\]
is irreducible. Since $q_1\not \in \supp(b_{z-1})$, $a\neq b_{z-1}$. Hence we have the non-confluent triple $(x,a,b_{z-1})$.
~\\

\noindent Case 2: $n>1$. Let $x=q_{1}+\alpha_{2}u$, $a= q_1+q_2+\Sigma_1+\Sigma_2$, and $b=n\Sigma_1+(n+1)q_1+d{u}$. 

As shown in the proof of Lemma \ref{lemma:nonconfluent1-revision}, $x\rightarrow^* a$ and $x\rightarrow^* b$. Clearly, $a$ is irreducible. 
Since $u\neq s_i$ for all $i$,  $d<\alpha_1$ and $d<(n-1)\alpha_1+d= \alpha_2-\alpha_1$, $c_u(b)=d<\mu(u)$. Since  $\supp_0(b)\setminus \{u\}\subseteq \textnormal{Sink}(E)$, it follows that $b$ is irreducible.  Clearly, $a\neq b$. Hence, we obtain a non-confluent triple $(x,a,b)$.

Therefore, $(F_\mathcal{E}, \rightarrow)$ is non-confluent. By Lemma \ref{lemma:diamondlemma-like}, $\mathcal{M}(E,w)$ is non-confluent. 
\end{proof}

The following corollary directly follows from the proof of Theorem \ref{theo:nonclulent_n-stargraph}.

\begin{corollary} 
    Let $(E,w)$ be a multi-weighted $n$-star graph with source vertex $u$ and weights $\alpha_1<\alpha_2<\dots <\alpha_m$. Writing $\alpha_2=n\alpha_1+d$, $0\leq  d<\alpha_1$, 
    then the following are non-confluent triples in $\mathcal{M}(E,w)$: 
    \[(x,a,b)=
    \begin{cases}
        \left (q_1+\alpha_1u,~~ 2q_1+\Sigma_1,~~\displaystyle\sum_{i=2}^z\Sigma_i+q_{z}+(2\alpha_1-\alpha_z)u\right)  & \text{if~}n=1\\
        \left (q_1+\alpha_2u,~~ q_1+q_2+\Sigma_1+\Sigma_2,~~(n+1)q_1+n\Sigma_1+du
        \right ) & \text{if~}n>1.
    \end{cases}
    \]
where $q_i:=q_i^u$, $\Sigma_i:=\Sigma_{\alpha_i}(u)$, and with $S:=\{2\leq j\leq    m-1~|~2\alpha_1<\alpha_{j+1}\}$, 

\[z=\begin{cases}
 \min S & \text{if~}S\neq \varnothing,\\
 m & \text{if~}S=\varnothing. 
\end{cases}
\]

\end{corollary}

We now consider weighted graphs which contain a weighted star subgraph; see \cite{preusser2021weighted} for the definition of a weighted subgraph. In Lemma \ref{lemma:nonconfluent1-revision}, the assumption that the source vertex $u$ of the multi-weighted star subgraph satisfies $u\not\in T(u)\setminus{u}$ is essential to the proof. This naturally raises the question of whether this assumption can be weakened. In particular, can the conclusion still hold under the weaker assumption $u\not\in r(s^{-1}(u))$?

\begin{example}\label{ex:confluent_case_1_not_satisfied}
    Consider the weighted graph below:
    
    
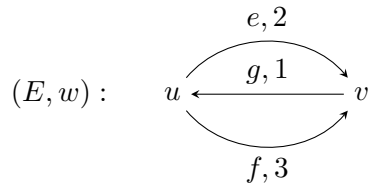
\begin{center}
\begin{tikzpicture}[>=stealth, scale=1, every node/.style={scale=1}]
   
    \node at (0.0,0) {$(E,w):$};

    \node (u)  at (1.5,0) {$u$};
    \node (v) at (4.0,0) {$v$};

    \draw[->, bend left=50] (u) to node[above] {$e,2$} (v);
    \draw[->] (v) to node[above] {$g,1$} (u);
    \draw[->, bend right=50] (u) to node[below] {$f,3$} (v);
\end{tikzpicture}
\end{center}
Then $u\not \in r(s^{-1}(u))$ but $u\in T(u)\setminus \{u\}$. The monoid $F_\mathcal{E}$ is generated by $\mathcal{E}=\{u,v,q\}$ where $q:=q^u_1$ and we have the reductions at $u$ and $v$: 
\[
 2u ~\xlongrightarrow{(\mathfrak{R}1(u))}~ q+v, \qquad  q+u~ 
\xlongrightarrow{(\mathfrak{R}2(u))}~ v, \qquad \textnormal{~and~} \qquad  v ~
\xlongrightarrow{(\mathfrak{R}(v))} ~u. 
\]
Then for $x=q+2u$, $x\longrightarrow^* 2q+v$ and $x\longrightarrow^* v+u$. Now, 
\[
a:= 2q+v ~\xlongrightarrow{(\mathfrak{R}(v))}~ 2q+u ~\xlongrightarrow{(\mathfrak{R}2(u))}~ q+v
\]
and 
\[b:= v+u ~\xlongrightarrow{(\mathfrak{R}(v))}~2u~~\xlongrightarrow{(\mathfrak{R}1(u))}~q+v.
\]
Thus, $(x,a,b)$ is not a non-confluent triple for $\mathcal{M}(E,w)$. 
\end{example}

Consider Case 1 of Lemma \ref{lemma:nonconfluent1-revision} where the source vertex $u$ is such that  $w_2(u)=w_1(u)+d$ $(0< d<w_i(u))$. 
This is true for Example \ref{ex:confluent_case_1_not_satisfied} and we have seen that only having the condition $u\not \in r(s^{-1}(u))$ is not sufficient. In fact, the exact elements $x$, $a$, and $b$ in Example \ref{ex:confluent_case_1_not_satisfied} are the exact elements considered in  Case 1 of the proof of Lemma \ref{lemma:nonconfluent1-revision}. Actually, we shall see in Example \ref{ex:confluent_case_1_not_satisfied_2} that $\mathcal{M}(E,w)$ is confluent. To establish this, we use the following classical result by Newman, commonly referred to in the modern literature as \emph{Newman's Lemma}:

\begin{lemma}[{\cite[Theorem~3]{Newman1942}}] \label{lem:Newman} Let $(A, \mapsto)$ be an abstract reduction system such that the following are satisfied: 

\begin{enumerate}
    \item [\textnormal{(1)}] Every chain of reductions in $(A,\mapsto)$ terminates. That is, for every chain $C:~x_1\mapsto x_2\mapsto \cdots \mapsto x_n \mapsto \cdots$, there exists an irreducible element $y$ such that  $C:~x_1\mapsto x_2\mapsto \cdots \mapsto x_n \mapsto \cdots \mapsto ~y $. 
    
    \item [\textnormal{(2)}] $(A, \mapsto)$ is locally confluent. That is, if $x\mapsto a$ and $x\mapsto b$, then there exists $c\in A$ such that $a\mapsto^* c$ and $b\mapsto^* c$.  
\end{enumerate} Then $(A, \mapsto)$ is confluent. 
\end{lemma}

\begin{example}\label{ex:confluent_case_1_not_satisfied_2}
    Consider the graph $(E,w)$ in Example \ref{ex:confluent_case_1_not_satisfied}. Every element of $F_\mathcal{E}$ is of the form $\alpha u +\beta v + \gamma q$ where $\alpha, \beta, \gamma\in \mathbb{N}$. If $\beta\neq 0$, using ($\mathfrak{R}(v)$) $\beta$-times,
    \[
    \alpha u +\beta v + \gamma q\longrightarrow^* (\alpha +\beta)u+\gamma q. 
    \] Now, 
    \[
    q+u ~\xrightarrow{(\mathfrak{R}2(u))}~v \xrightarrow{(\mathfrak{R}(v))} ~u. \]
    Thus for  $\lambda, \tau>0$,  
\[\lambda u+\tau q  ~ \xrightarrow{(\mathfrak{R}2(u))}~ (\lambda -1)u+(\tau -1)q+v ~\xrightarrow{(\mathfrak{R}(v))}~\lambda u+(\tau -1)q~\longrightarrow\dots \longrightarrow \lambda u.  \]
Now for $\lambda \geq 2$, 
\[
\lambda u =
(\lambda-2)u+2u~\xrightarrow{(\mathfrak{R}2(u))}~(\lambda-2)u+v~\xrightarrow{(\mathfrak{R}(v))} (\lambda-2)u+u =(\lambda -1)u ~\longrightarrow\dots \longrightarrow  u.  
\]
Moreover, 
$\epsilon v\longrightarrow^* \epsilon u$ for any $\epsilon>0$. Hence, for any $\alpha, \beta, \gamma\in \mathbb{N}$, 
\[
\alpha u +\beta v + \gamma q~\longrightarrow^*~ c =\begin{cases}
    u & \text{if~}\alpha \neq 0\textnormal{~or~}\beta \neq 0, \\
    \gamma q & \textnormal{otherwise}.
\end{cases} 
\]
Define the lexicographic measure $\mathfrak{m}: F_{\mathcal{E}}\rightarrow \mathbb{N}^3$ given by $\mathfrak{m}(x)=(\alpha+\beta, \gamma +\beta, \gamma)$ for any $x=\alpha u +\beta v + \gamma q$. It is clear that under the lexicographic order on $\mathbb{N}^3$,  $\mathfrak{m}$ is strictly decreasing for every reduction. Hence,  there is no infinite chain of reductions in $(F_\mathcal{E}, \rightarrow)$. In particular, every chain of reductions in $(F_\mathcal{E}, \rightarrow)$ terminates to either $u$ or $\gamma q$, which are irreducible.

Let $x, a, b\in F_{\mathcal{E}}$ such that 
$x\xlongrightarrow{(R_1)} a$ and $x\xlongrightarrow{(R_2)} b$ for some reductions $(R_1), (R_2)\in \{ \mathfrak{R}1(u), \mathfrak{R}2(u), \mathfrak{R}(v)\}$. By symmetry, we have either $v\in \supp(a)\cap \supp(b)$, or $v\in \supp(a)$ and $u\in \supp(b)$. In both cases, $a\longrightarrow^* u$ and $b\longrightarrow^* u$. Hence, $(F_\mathcal{E}, \rightarrow)$ is locally confluent. By Newman's Lemma \ref{lem:Newman}, $(F_\mathcal{E}, \rightarrow)$ is  confluent. 
 
\end{example}

\begin{proposition}\label{prop:confluent_iff_beta-alpha_geq1}
Let $(E,w)$ be a weighted graph of the form 
\begin{center}
\begin{tikzpicture}[>=stealth, node distance=3cm, baseline=(u.base)]
    \node (E) at (-3.2,0) {$(E,w):$};
    \node (s1) at (-2,0) {$s_1$};
    \node (u)  at (0,0) {$u$};
    \node (s2) at (2,0) {$s_2$};

    \draw[->] (u) -- node[above] {$e_1,\alpha$} (s1);
    \draw[->] (u) -- node[above] {$e_2,\beta$} (s2);

\draw[->, bend right=50] (s1) to node[below] {$f_1,1$} (u);

\draw[->, bend left=50] (s2) to node[below] {$f_2,1$} (u);
\end{tikzpicture}
\end{center}
where $\alpha\leq \beta $. Then $\mathcal{M}(E,w)$ is confluent if and only if $\alpha=1$ or  $\beta-\alpha\leq 1$. 
\end{proposition}

\noindent\emph{Sketch of Proof:}  We first outline the proof, after which we provide the details. 
\begin{enumerate}
    \item Suppose that $\beta-\alpha\leq 1$. Then
    $\mathcal{M}(E,w)$ is confluent.
    
    \begin{enumerate}[leftmargin=2cm]
        \item [Case 1:] $\beta-\alpha=0$.
        
        \item [Case 2:] $\beta-\alpha=1$.
        
        \begin{enumerate}[leftmargin=2cm]
            \item [Subcase 2.1:] $\alpha=1$.
            
            \item [Subcase 2.2:] $\alpha>1$. Utilizing  Newman's Lemma \ref{lem:Newman}, we
            prove the following:
            
            \begin{enumerate}
                \item [Claim 1:] Every reduction chain in
                $F_\mathcal{E}$ terminates.
                
                \item [Claim 2:] $(F_\mathcal{E},\rightarrow)$ is
                locally confluent.
            \end{enumerate}
        \end{enumerate}
    \end{enumerate}
    
    \item Suppose that $\beta-\alpha>1$.
    
    \begin{enumerate}[leftmargin=2cm]
        \item [Case 1:] $\alpha>1$. Then
        $\mathcal{M}(E,w)$ is non-confluent.
        
        \begin{enumerate}[leftmargin=2cm]
            \item [Subcase 1.1:] $\beta-\alpha\leq\alpha$.
            
            \item [Subcase 1.2:] $\beta-\alpha>\alpha$.
        \end{enumerate}
        
        \item [Case 2:] $\alpha=1$. Then
        $\mathcal{M}(E,w)$ is confluent.
    \end{enumerate}
\end{enumerate}

\begin{proof} Suppose $\beta-\alpha \leq 1$.

\noindent Case 1: $\beta-\alpha =0$. Then $(E,w)$ is vertex-weighted. By Lemma \ref{lem:confluence_vertex-weighted_graphs}, $\mathcal{M}(E,w)$ is confluent. Assume that $\beta -\alpha>0$. Then $\mathcal{E}=\{u, s_1, s_2, q\}$ where $q:= q_1^u$ and the reductions at $u$ and $v$ are 
\[
 \alpha u ~\xlongrightarrow{(\mathfrak{R}1(u))}~ q+s_1, \qquad  q+(\beta-\alpha)u~ 
\xlongrightarrow{(\mathfrak{R}2(u))}~ s_2, \qquad s_1 ~
\xlongrightarrow{(\mathfrak{R}(s_1))} ~u,~  \textnormal{~and~} \qquad  s_2 ~
\xlongrightarrow{(\mathfrak{R}(s_2))} ~u. 
\]
Let $\varphi=c_u+c_{s_1}+c_{s_2}:F_{\mathcal{E}}\longrightarrow \mathbb{N}$ and $x\in F_\mathcal{E}$. If $\varphi(x)\neq 0$, then by $(\mathfrak{R}(s_1))$ and $(\mathfrak{R}(s_2))$, we have 
\begin{equation}\label{eqn:reduced}
x\longrightarrow^*\varphi(x)u+c_q(x)q.\tag{$*$}
\end{equation}
~\\
\noindent Case 2: $\beta - \alpha =1$. 
If $\varphi(x), c_q(x)\neq 0$, then 
\[x\longrightarrow^*\varphi(x)u+c_q(x)q  \xrightarrow{(\mathfrak{R}2(u))} (\varphi(x) -1)u+(c_q(x)-1)q+s_2 \xrightarrow{(\mathfrak{R}(s_2))} \varphi(x) u + (c_q(x) -1)q. 
\]
Repeating these reductions, we obtain $x\longrightarrow^* \varphi(x)u$. Hence, in this case for general $x\in F_\mathcal{E}$ with $\varphi(x)\neq 0$ ($c_q(x)$ could be $0$), $x\longrightarrow^* \varphi(x)u$.  \\

\noindent Subcase 2.1: $\alpha=1$. Then $\beta =2$. To prove confluence, we first show that  $\varphi$ is invariant under reductions, that is for any $a$ with $x\longrightarrow^* a$, $\varphi(x)=\varphi(a)$. Since every reduction is a composition of  $(\mathfrak{R}1(u)), (\mathfrak{R}2(u)), (\mathfrak{R}(s_1))$ and $(\mathfrak{R}(s_2))$, it suffices to verify the invariance of $\varphi$ for each of them. Let $x=\lambda u +\kappa s_1+\gamma s_2+\tau q\in F_{\mathcal{E}}$ for some $\lambda, \kappa, \gamma,\tau \in \mathbb{N}$. Then $\varphi(x)=\lambda +\kappa + \gamma.$

If $x\xlongrightarrow{(\mathfrak{R}1(u))} a$, then $\lambda\geq 1$ and
\[
x\xlongrightarrow{(\mathfrak{R}1(u))} a = (\lambda - 1)u +(\kappa +1)s_1+\gamma s_2+(\tau +1)q.
\]
Thus,  $\varphi(a)= \lambda -1 +\kappa+1+\gamma= \lambda +\kappa + \gamma$.

If $x\xlongrightarrow{(\mathfrak{R}2(u))} a$, then $\lambda,\tau \neq 0$ and 
\[
x\xlongrightarrow{(\mathfrak{R}2
(u))} a = (\lambda - 1)u +\kappa s_1+(\gamma+1) s_2+(\tau -1)q.
\]
Thus, $
\varphi(a)= \lambda -1 +\kappa+\gamma+1=\lambda+\kappa+\gamma$.

If  $x\xlongrightarrow{(\mathfrak{R}(s_1))} a$, then $\kappa \neq 0$ and 
\[
x\xlongrightarrow{(\mathfrak{R}
(s_1))} a = (\lambda+1) u +(\kappa-1) s_1+\gamma s_2+\tau q.
\]
Thus, $
\varphi(a)= \lambda +1 +\kappa-1+\gamma=\lambda+\kappa+\gamma$.

If  $x\xlongrightarrow{(\mathfrak{R}(s_2))} a$, then the previous argument applies similarly. \\

By (\ref{eqn:reduced}),  $x\longrightarrow^*\varphi(x) u + c_q(x) q$. If $\varphi(x)=0$, then $x=c_q(x) q$ and hence $x$ is irreducible. If  $\varphi(x) \neq 0$, then $x\longrightarrow^*\varphi(x) u + c_q(x) q\longrightarrow^*\varphi(x)u$. Let $a,b\in F_{\mathcal{E}}$ such that $x\longrightarrow^* a$ and $x\longrightarrow^* b$. Since $\varphi$ is invariant under reductions, $\varphi(a)=\varphi(b)=\varphi(x)$. If $\varphi(x)=0$, then $x$ is irreducible and hence, $a=b=x$. If $\varphi(x)\neq 0$, then $a\longrightarrow^* \varphi(a)u=\varphi
 (x)u$ and $b\longrightarrow^* \varphi(b)u=\varphi(x)u$. \\

\noindent Subcase 2.2: $\alpha>1$. We shall use Newman's Lemma \ref{lem:Newman} to show confluence. \\

\noindent Claim 1: Every chain of reductions in $F_{\mathcal{E}}$ terminates. Define the lexicographic measure $\mathfrak{m}:F_\mathcal{E}\rightarrow \mathbb{N}^3$ given by $\mathfrak{m}(x)=(N_u(x), N_q(x), c_q(x))$ for any element $x=\lambda u+\kappa s_1+\gamma s_2+\tau q\in F_\mathcal{E}$ where $N_t=c_t+c_{s_1}+c_{s_2}$, $t=u,q$. It could be verified that under the lexicographic order on $\mathbb{N}^3$,  $\mathfrak{m}$ is strictly decreasing for every reduction. Hence, there is no infinite chain of reductions in $(F_\mathcal{E}, \rightarrow)$.

Let $x\in F_\mathcal{E}\setminus \{0\}$. If $\varphi(x)= 0$, then $x=c_q(x)q$, an irreducible element. If $\varphi(x)\neq 0$, then $x\longrightarrow^* \varphi(x)u$. If $\varphi(x) < \alpha$, then $\varphi(x) u $ is irreducible. Suppose $\varphi(x) \geq \alpha$. Then $\varphi(x) -1 \geq \alpha-1>0$. Thus,  $\varphi(x)-1=n(\alpha-1)+d$ for some $n\geq 1$ and $0\leq d<\alpha-1$. Thus, $\varphi(x) = n(\alpha-1)+D_x$ where $D_x=d+1$. Hence, $1\leq D_x\leq \alpha-1$. Now, 
\begin{equation*}
\begin{split}
\varphi(x) u = ~& (\varphi(x) -\alpha)u + \alpha u  \xrightarrow{(\mathfrak{R}1(u))}  (\varphi(x) -\alpha)u + (q+s_1)  \xrightarrow{(\mathfrak{R}(s_1))} (\varphi(x) -\alpha)u + (q+u) \\
& \xrightarrow{(\mathfrak{R}2(u))}(\varphi(x) -\alpha)u +s_2   \xrightarrow{(\mathfrak{R}(s_2))}  (\varphi(x) -\alpha)u +u   = (\varphi(x) -(\alpha-1))u. 
\end{split}
\end{equation*}
Hence, $\varphi(x)u \longrightarrow^* (\varphi(x) - n(\alpha-1))u= D_xu$. Accordingly, for any $x\in F_\mathcal{E}\setminus \{0\}$, 
\[
x~\longrightarrow^*~ c =\begin{cases}
    c_q(x)q & \textnormal{~if~} \varphi(x)= 0, \\ 
    D_xu & \textnormal{~otherwise},
\end{cases}
\] where $D_x$ is the unique integer with $1\leq D_x\leq \alpha -1$ for which  
$\varphi(x) \equiv D_x \pmod{\alpha -1}$. Hence, every chain of reductions in $(F_\mathcal{E}, \rightarrow)$ terminates to either $Du$ for some $D<\alpha$, or $\tau q$ for some $\tau\in \mathbb{N}$, which are irreducible elements in $F_{\mathcal{E}}$.\\

\noindent Claim 2: $(F_\mathcal{E}, \rightarrow)$ is locally confluent. 
Let $x=\lambda u +\kappa s_1 + \gamma s_2+ \tau q\in F_{\mathcal{E}} $ and $ a \in F_{\mathcal{E}}$ such that 
$x\xlongrightarrow{(R)} a$ for some reduction $R= \mathfrak{R}1(u), \mathfrak{R}2(u), \mathfrak{R}(s_1),  \mathfrak{R}(s_2)$. Then $\varphi(x)=\lambda +\kappa +\gamma\neq 0$. It is enough to show that $\varphi(a) \equiv D_x \pmod{\alpha -1}$.

\noindent If $x\xlongrightarrow{(\mathfrak{R}1(u))} a$, then $\lambda\geq \alpha$ and
\[
x\xlongrightarrow{(\mathfrak{R}1(u))} a = (\lambda - \alpha)u +(\kappa +1)s_1+\gamma s_2+(\tau +1)q.
\]
Thus, \[
\varphi(a)= \lambda -\alpha +\kappa+1+\gamma= \lambda +\kappa + \gamma -(\alpha -1) \equiv D_x-(\alpha -1)\equiv D_x \pmod{\alpha -1}.\]

\noindent If $x\xlongrightarrow{(\mathfrak{R}2(u))} a$, then $\lambda,\tau \neq 0$ and similarly above, one can compute that
\[
\varphi(a)= \lambda+\kappa+\gamma\equiv D_x \pmod{\alpha -1}.\]

\noindent If  $x\xlongrightarrow{(\mathfrak{R}(s_1))} a$, then $\kappa \neq 0$ and  

\[
\varphi(a)=\lambda+\kappa+\gamma\equiv D_x \pmod{\alpha -1}.\]
\noindent If  $x\xlongrightarrow{(\mathfrak{R}(s_2))} a$, then the previous computations apply similarly.\\

Hence, $(F_\mathcal{E}, \rightarrow)$ is locally confluent. By Newman's Lemma \ref{lem:Newman}, $(F_\mathcal{E}, \rightarrow)$ is  confluent, and so is $\mathcal{M}(E,w)$.\\

Now, we assume that $d:=\beta -\alpha >1$ and show that $\mathcal{M}(E,w)$ is non-confluent for $\alpha>1$ and confluent for $\alpha=1$. In this case, the reductions at $u$ and $v$ are 
\[
 \alpha u ~\xlongrightarrow{(\mathfrak{R}1(u))}~ q+s_1, \qquad  q+du~ 
\xlongrightarrow{(\mathfrak{R}2(u))}~ s_2, \qquad s_1 ~
\xlongrightarrow{(\mathfrak{R}(s_1))} ~u,~  \textnormal{~and~} \qquad  s_2 ~
\xlongrightarrow{(\mathfrak{R}(s_2))} ~u. 
\]

\noindent Case 1: $\alpha>1$. \\
\noindent Subcase 1.1: $2\leq d\leq \alpha$. Let $x=q+\alpha u$. Then 
\[
x=q+\alpha u \xrightarrow{(\mathfrak{R}1(u))} q+q+s_1 \xrightarrow{(\mathfrak{R}(s_1))} q+q+u =2q+u := a. 
\] 
Since $c_u(a)=1<\alpha, d$, it follows that $a$ is irreducible. Now, 
\[
x=q+\alpha u =q+du+(\alpha-d)u \xrightarrow{(\mathfrak{R}2(u))} s_2 +(\alpha -d)u \xrightarrow{(\mathfrak{R}(s_2))} u+(\alpha -d)u=(\alpha -d+1)u:= b. 
\] Since $c_u(b)=\alpha -d+1< \alpha$, $b$ is irreducible. Since $q\not \in \supp(b)$, $a\neq b$. Thus, we obtain a non-confluent triple
\[
(q+\alpha u,~  2q+u,~ (\alpha-d+1)u). 
\]

\noindent Subcase 1.2: $d> \alpha$. Let $x=q+du$. Then 
\[
x=q+du \xrightarrow{(\mathfrak{R}2(u))} s_2 \xrightarrow{(\mathfrak{R}(s_2))} u:=a 
\]

Since $c_u(a)=1<\alpha$, $a$ is irreducible. We have $d=n(\alpha -1)+D$ for some $1\leq D\leq \alpha -1$. 
Hence, 
\[d u = (d -\alpha)u + \alpha u  \xrightarrow{(\mathfrak{R}1(u))}  (d -\alpha)u + (q+s_1)  \xrightarrow{(\mathfrak{R}(s_1))}  (d -\alpha)u + q+u  =
q+(d-(\alpha-1))u.
\]
Repeating the reductions, we obtain
$du\longrightarrow^* nq+(d - n(\alpha-1))u= nq+Du.  $
Thus, $x=q+du\longrightarrow^* (n+1)q+Du:=b$  where $D$ is the unique integer for which  
$ d \equiv D \pmod{\alpha -1}$ with $1\leq D\leq \alpha -1$. Since $c_u(b)=D<\alpha,d$, it follows that $b$ is irreducible. Since $q\not \in \supp(a)$, $a\neq b$. Hence, we obtain a non-confluent triple 
\[
(q+du,~u,~(n+1)q+Du). 
\]

 \noindent Case 2: $\alpha=1$. Let $\varphi'=c_u+c_{s_1}+c_{s_2}:F_{\mathcal{E}}\longrightarrow \mathbb{Z}_{d-1}$. That is,  $\varphi'=r\circ \varphi$ where $r$ takes the residue modulo $d-1$. To prove confluence, we first show that $\varphi'$ is invariant under reductions, that is for any $a$ with $x\longrightarrow^* a$, $\varphi'(x)\equiv\varphi'(a)\pmod{d-1}$.

 Let $x=\lambda u +\kappa s_1+\gamma s_2+\tau q$ for some $\lambda, \kappa, \gamma,\tau \in \mathbb{N}$. Then $\varphi'(x)=\lambda +\kappa + \gamma\pmod{d-1}$.

\noindent Subcase 2.1: $x\xlongrightarrow{(\mathfrak{R}1(u))} a$. Then $\lambda\geq 1$ and
\[
x\xlongrightarrow{(\mathfrak{R}1(u))} a = (\lambda - 1)u +(\kappa +1)s_1+\gamma s_2+(\tau +1)q.
\]
Now, \[
\varphi'(a)= \lambda -1 +\kappa+1+\gamma= \lambda +\kappa + \gamma=\varphi'(x).\]

\noindent Subcase 2.2: $x\xlongrightarrow{(\mathfrak{R}2(u))} a$. Then $\lambda,\tau \neq 0$ and 
\[
x\xlongrightarrow{(\mathfrak{R}2
(u))} a = (\lambda - d)u +\kappa s_1+(\gamma+1) s_2+(\tau -1)q.
\]
Now, \[
\varphi'(a)= \lambda -d +\kappa+\gamma+1=\lambda+\kappa+\gamma  -(d-1)\equiv \varphi'(x)\pmod{d-1}.\]
\noindent Subcase 2.3: $x\xlongrightarrow{(\mathfrak{R}(s_1))} a$. Then $\kappa \neq 0$ and 
\[
x\xlongrightarrow{(\mathfrak{R}
(s_1))} a = (\lambda+1) u +(\kappa-1) s_1+\gamma s_2+\tau q.
\]
Now, \[
\varphi'(a)= \lambda +1 +\kappa-1+\gamma=\lambda+\kappa+\gamma=\varphi'(x).\]
\noindent Subcase 2.4: $x\xlongrightarrow{(\mathfrak{R}(s_2))} a$. The argument in Subcase 2.3 applies similarly to Subcase 2.4.\\

  Let $x\in F_\mathcal{E}\setminus \{0\}$. If $\varphi(x)=0$, then $x=\tau q$ for some $\tau\in \mathbb{N}^*$, an irreducible element. Hence, we assume that $\varphi(x)>0$. Then by (\ref{eqn:reduced}), $(\mathfrak{R}1(u))$, and $(\mathfrak{R}(s_1))$, we have 
\[
x\longrightarrow^* \varphi(x)u +c_q(x)q\longrightarrow^* \varphi(x)u +(c_q(x)+1)q.
\]
Hence, $x\longrightarrow^*  \varphi(x)u + \tau q$ for any $\tau\geq c_q(x)$.
Now, whenever $\varphi(x)\geq d$, 
\[
x\longrightarrow^* \varphi(x)u + (c_q(x)+1)q \xrightarrow{(\mathfrak{R}2(u))}(\varphi(x) -d)u + c_q(x)q+ s_2\xrightarrow{(\mathfrak{R}(s_2))}(\varphi(x) -(d-1))u + c_q(x)q. 
\]
Repeating these reductions, we obtain 
\[
x\longrightarrow^* (\varphi(x) -(d-1))u + c_q(x)q \longrightarrow^* D'u+ c_q(x)q
\]
where $D'$ is the unique integer such that $1\leq D'\leq d-1$ and $\varphi(x)=m(d-1)+D'$. That is, $D'\equiv \varphi'(x)\pmod{d-1}$. Now, let $a,b\in F_\mathcal{E}$ such that $x\longrightarrow^* a$ and $x\longrightarrow^* b$. Then $\varphi'(a)\equiv \varphi'(b)\equiv \varphi'(x)\pmod{d-1}$. Take $z=\max\{c_q(a), c_q(b)\}$. Thus,  
\[
a\longrightarrow^* \varphi(a) u +c_q(a)q\longrightarrow^* D'u+ c_q(a)q\longrightarrow^* D'u+ zq\] and \[ b\longrightarrow^* \varphi(b) u +c_q(b)q\longrightarrow^* D'u+ c_q(b)q\longrightarrow^* D'u+ zq. 
\]\end{proof}

\begin{corollary}
Let $(E,w)$ be the weighted graph described in Proposition 
\ref{prop:confluent_iff_beta-alpha_geq1}
 with $\alpha,\beta-\alpha>1$. Then the following are non-confluent triples in $\mathcal{M}(E,w)$: 
\[
(x,a,b)=\begin{cases}
  (q+\alpha u,~  2q+u,~ (2\alpha -\beta +1)u) &  \textnormal{~if~} \beta \leq 2\alpha, \\
  (q+du,~u,~(n+1)q+Du) & \textnormal{~otherwise}, 
\end{cases}
\]
where $n$ and $D$ are the unique positive integers for which $\beta-\alpha = n(\alpha-1)+D$, $1\leq D\leq \alpha -1$.
\end{corollary}

In view of Lemma~\ref{lemma:nonconfluent1-revision} and Proposition~\ref{prop:confluent_iff_beta-alpha_geq1}, the condition that the source vertex $u$ of the multi-weighted star subgraph satisfies $u\notin T(u)\setminus\{u\}$ is sufficient to establish non-confluence, whereas weakening this to the condition $u\notin r(s^{-1}(u))$ alone is not. In the following proposition, for a weighted graph having a certain multi-weighted star subgraph, we impose additional conditions to having $u\notin r(s^{-1}(u))$ on the graph to ensure that the resulting reductions cannot be joined, thereby establishing non-confluence.

For the following result, recall Notation~\ref{nota:Sigma(u)} and Definition~\ref{def:c_s,mu(s),support}.

\begin{proposition}\label{prop:non-confluent2}
Let $(E,w)$ be a weighted graph. Suppose there exists $u\in E^0$ which emits edges of exactly two distinct weights $
\alpha<\beta $ and $u\not \in r(s^{-1}(u)) $. 
Write
$\beta=m\alpha+\delta$ where $m\geq 1$ and
$0\leq \delta<\alpha$ and assume that for every $s\in\supp_0(\Sigma_\alpha(u))\cup\supp_0(\Sigma_\beta(u))$,
the following condition holds:
\[
\mu(s)>
\begin{cases}
\max\{ c_s(\Sigma_\alpha(u)), c_s(\Sigma_\beta(u))\}
&\text{~if~} m=1,\\[4pt]
\max\{c_s(\Sigma_\alpha(u)) + c_s(\Sigma_\beta(u)),\,m c_s(\Sigma_\alpha(u))\}
& \text{~otherwise}.
\end{cases}
\]
Then $\mathcal{M}(E,w)$ is non-confluent.
\end{proposition}

\begin{proof}
We shall construct an explicit non-confluent triple for $\mathcal{M}(E,w)$. For $\gamma=\alpha, \beta$, let $\Sigma_\gamma:=\Sigma_\gamma(u)$. Moreover, let $q:=q_1^u$. Then the reductions at $u$ are
\[
(\mathfrak{R}1)~~
\alpha u\longrightarrow q+\Sigma_\alpha
\textnormal{~~~~and~~~~}
(\mathfrak{R}2)~~
q+(\beta-\alpha)u\longrightarrow \Sigma_\beta.
\]
Now, $
\beta=m\alpha+\delta$ for some
$
m\geq1$ and   $0\leq\delta<\alpha$.\\

\noindent Case 1: $m=1$. Then $\beta = \alpha +\delta$ with $1\leq \delta<\alpha $, that is,  $\alpha <\beta <2\alpha $. Let $x:= q+\alpha u$. Then 

\[    x = q+\alpha u \xlongrightarrow{(\mathfrak{R}1)} q+q+\Sigma_\alpha = 2q+\Sigma_\alpha := a. 
\] 
Note that $a$ is irreducible since for $s\in\supp_0(a)= \supp_0(\Sigma_\alpha(u))$, $\mu(s)>c_s(\Sigma_\alpha(u))=c_s(a)$. 
\\

Writing $\alpha=(\beta-\alpha)+(2\alpha-\beta)$, 
\[
x= q+\alpha u = q+(\beta-\alpha)u+(2\alpha-\beta)u \xrightarrow{(\mathfrak{R}2)} \Sigma_\beta + (2\alpha -\beta)u:=b
\]
Note that $b\in F_{E^0}$. Let $s\in \supp_0(b)$. If $s=u$, then $c_u(b)=2\alpha-\beta <\alpha=w_1(u)$ since $u\not \in r(s^{-1}(u))$ and since $m=1$. If $s\neq u$, then $c_s(b)=c_s(\Sigma_\beta)<\mu(s)\leq w_1(s)$. Thus, $b$ is irreducible.

Clearly, $a\neq b$. Hence, we obtain a non-confluent triple
\[
(q+\alpha u, 2q+\Sigma_\alpha, \Sigma_\beta+(2\alpha-\beta)u). 
\]
Case 2: $m\geq 2$. Let $x:= q+\beta  u$. Then 
\[
x
=
q+(\beta-\alpha)u+\alpha u
\xlongrightarrow{(\mathfrak{R}2)}
\Sigma_\beta+\alpha u
\xlongrightarrow{(\mathfrak{R}1)}
\Sigma_\beta+q+\Sigma_\alpha:= a. 
\] 
Note that $a$ is irreducible since for $s\in \supp_0(a)=\supp_0(\Sigma_\alpha)\cup\supp_0(\Sigma_\beta)$, $c_s(a)=c_s(\Sigma_\alpha)+c_s(\Sigma_\beta)<\mu(s)$.  \\

Now, $x=q+\beta u = q+(m\alpha +\delta)u=q+m\alpha u +\delta u$. Applying $(\mathfrak{R}1)$ $m$-times, 
\[
x
\longrightarrow^*
q+m(q+\Sigma_\alpha)+\delta u=
(m+1)q+m\Sigma_\alpha+\delta u:=b. 
\]
Let $s\in \supp_0(b)$. For $s=u$, we have $c_u(b)=\delta$ since $u\not \in r(s^{-1}(u))$. We have $\delta <\alpha$ and $\delta < (m-1)\alpha +\delta=  m\alpha +\delta - \alpha=
\beta -\alpha$ since $m\geq 2$. 
Thus, $c_u(b)<\min\{\alpha, \beta-\alpha \}= \mu(u)$. For $s\neq u$, $c_s(b)=mc_s(\Sigma_\alpha)<\mu(s)$, by assumption. Hence, $b$ is irreducible.

Clearly, $a$ and $b$ are distinct. Hence we obtain a non-confluent triple 
\[(q+\beta u, q+\Sigma_\beta+\Sigma_\alpha, (m+1)q+m\Sigma_\alpha+\delta u).\] 
Therefore, $\mathcal{M}(E,w)$ is non-confluent.\end{proof}

\begin{remark}
Consider again Example \ref{ex:confluent_case_1_not_satisfied}. Notice that for the vertex $v$, $\mu(v)=1=c_v(v)=c_v(\Sigma_2(u))= c_v(\Sigma_3(u))$ and indeed, the assumption of Proposition \ref{prop:non-confluent2} is not satisfied. 

\end{remark}

\begin{corollary}
    Let $(E,w)$ be the weighted graph described in Proposition \ref{prop:non-confluent2}. The following are  non-confluent triples in  $\mathcal{M}(E,w)$: 
    \[
(x,a,b)=\begin{cases}
 (q+\alpha u,~~2q+\Sigma_\alpha(u),~~\Sigma_\beta(u)+(2\alpha-\beta)u) &  ~\text{~if~} m=1, \\
    (q+\beta u,~~q+\Sigma_\beta(u)+\Sigma_\alpha(u),~~(m+1)q+m\Sigma_\alpha(u)+\delta u) &  ~\text{~otherwise},
\end{cases}
\]
where $q:=q^u_1$.

\end{corollary}

\begin{example}\label{ex:PropConditions_not_enough}
    Consider the weighted graph $(E,w)$ described in Proposition~\ref{prop:confluent_iff_beta-alpha_geq1} and take $\alpha=2$ and $\beta=4$. Then $\mathcal{M}(E,w)$ is non-confluent. However, one can verify that the additional condition of Proposition \ref{prop:non-confluent2} is not satisfied, showing that Proposition 
   \ref{prop:non-confluent2}'s hypothesis, while sufficient, does not give a complete characterization of non-confluence.

\end{example}

\begin{conclusion}
    
We characterized weighted graph monoid confluence on certain weighted graphs such as multi-weighted $n$-star graphs (Theorem~\ref{theo:nonclulent_n-stargraph}) and graphs described in Proposition~\ref{prop:confluent_iff_beta-alpha_geq1}. However, as we have seen, providing a full characterization of confluence in general weighted graph monoids appears to require progressively more elaborate hypotheses. Lemma~\ref{lemma:nonconfluent1-revision} provides a sufficient condition for non-confluence, but Proposition~\ref{prop:confluent_iff_beta-alpha_geq1} demonstrates that this condition is not necessary, and hence does not provide a complete characterization.
Addressing this, Proposition~\ref{prop:non-confluent2} establishes a sufficient condition for non-confluence under weaker structural assumptions at the expense of additional technical hypotheses on the weighted graph. Example~\ref{ex:confluent_case_1_not_satisfied_2} shows that failing one of these hypotheses can yield a confluent weighted graph monoid, while Example~\ref{ex:PropConditions_not_enough} shows that a weighted graph monoid may fail to satisfy these hypotheses and yet still be non-confluent.
\end{conclusion} 
Therefore, it is a good time to stop here and conclude this section with the following open problem for further investigation.

\begin{openproblem}
Let $(E, w)$ be an arbitrary row-finite weighted graph. Give necessary and sufficient conditions on $(E, w)$ for the weighted graph monoid $\mathcal{M}(E, w)$ to be confluent.
\end{openproblem}
\medskip

\section{Weighted Graph Monoid of (LPA) Weighted graphs}\label{sec:WLPAof(LPA)weightedgraphs}

Let \((E,w)\) be a weighted graph satisfying Condition \((\mathrm{LPA})\), which we shall call an \emph{(LPA) weighted graph}.
By Preusser's computational procedure \cite{Preusser2019WeightedLP}, one obtains an unweighted directed graph
\(F\) such that as $\K$-algebras, 
\[
L_\K(E,w)\cong L_\K(F).
\]
This is performed in two steps we shall call \emph{Step $\mathcal{P}1$} and \emph{Step $\mathcal{P}2$}:

\[
(E,w)\quad \xlongrightarrow{\textnormal{Step~} \mathcal{P}1} \quad (\widetilde{E},\widetilde{w})\quad \xlongrightarrow{\textnormal{Step~} \mathcal{P}2}\quad F, \]
where $(\widetilde{E}, \widetilde{w})$ is a new weighted graph obtained from $(E,w)$. In particular, $L_\K(E,w)\cong L_\K(\widetilde{E},\widetilde{w})\cong L_\K(F)$; see Lemma \ref{lem:P1P2isomorphic}. From $(\widetilde{E},\widetilde{w})$, we define an auxiliary unweighted graph $\widetilde{E}_{\rm{aux}}$ which we shall utilize to characterize cancellativity of weighted graph monoids of (LPA) weighted graphs.

Recall that for $u,v\in E^0$, we write $u\geq v$ if there exists a path $p\in\textnormal{Path}(E)$ such that $s(p)=u$ and $r(p)=v$. For a vertex $v\in E^0$, the \emph{tree of $u$} is the set $T(u)=\{w\in E^0~|~u\geq w\}$. For $X\subseteq E^0$, the \emph{tree of $X$} is set $T(X)= \bigcup_{v\in X} T(v)$. Two edges $e,f\in E^1$ are said to be \emph{in line} if $e=f$, $r(e)\geq s(f)$ or $r(f)\geq s(e)$.  

\subsection{Preusser's Two-Step Construction from Weighted to Unweighted Graph}\label{subsec:2-Step_(LPA)Graphs}

For a weighted graph $(E,w)$, denote by $E_w=\{e\in E^1: w(e)>1\}$ and $E_{uw}=\{e\in E^1: w(e)=1\}$, are  the sets of weighted and  unweighted edges in $(E,w)$, respectively. 

In \cite{Preusser2019WeightedLP}, Preusser provided a set of conditions which characterizes when a weighted Leavitt path algebra is isomorphic to a Leavitt path algebra.

\begin{definition}[(LPA) weighted graphs] 
We say that a weighted graph $(E,w)$ satisfies \emph{Condition (LPA)} if the following holds true:
\begin{enumerate}[leftmargin=1.5cm]
    \item [(LPA1)]
    Any vertex $v\in E^0$ emits at most one weighted edge.
    \item [(LPA2)]
    Any vertex $v\in T(r(E_w))$ emits at most one edge.
    \item [(LPA3)] If two weighted edges $e,f\in E_w$ are not in line, then $T(r(e))\cap T(r(f))=\varnothing$.  
    
    \item [(LPA4)]
    If $e\in E_w$ and $C$ is a cycle based at some vertex $v\in T(r(e))$, then $e$ is an edge in $C$.
\end{enumerate}
If $(E,w)$ satisfies Condition (LPA), then we say $(E,w)$ is an \emph{(LPA) weighted graph}. 
\end{definition}

\

\begin{theorem}[{\cite[Theorems 1 and 2]{Preusser2019WeightedLP}}]\label{theo:WLPA=LPA_iff_LPAGraph} 
    Let $(E,w)$ be a row-finite weighted graph and $\K$ a field. Then $L_\K(E,w)$ is isomorphic to an unweighted Leavitt path algebra if and only if $(E,w)$ satisfies Condition (LPA).
\end{theorem}

\begin{definition}[Step $\mathcal{P}1$] 
Let $(E,w)$ be an (LPA) weighted graph and set $Z:=T(r(E_w))$. Construct the weighted graph $(\widetilde{E},\widetilde{w})$ by replacing each edge $e\in s^{-1}(Z)$ by $w(e)$ copies of unweighted edges with reversed orientation. 

More formally, for an (LPA) weighted graph $(E,w)$, $(\widetilde{E},\widetilde{w})$ is the weighted graph defined as follows: $(\widetilde{E})^{0} = E^{0}$,
$(\widetilde{E})^{1} = (\widetilde{E})^{1}_{Z} \sqcup (\widetilde{E})^{1}_{Z^{c}}$ where
\[ (\widetilde{E})^{1}_{Z}
= \left\{
e^{(1)},\ldots,e^{(w(e))}
\;\middle|\;
e\in E^{1},\ s(e)\in Z
\right\},\qquad 
(\widetilde{E})^{1}_{Z^{c}}
= \left\{
e\mid e\in E^{1},\ s(e)\notin Z
\right\},\]
\[
\begin{aligned}
\widetilde{s}(e^{(i)}) &= r(e),\qquad
\widetilde{r}(e^{(i)}) = s(e),\qquad
\widetilde{w}(e^{(i)}) = 1
\quad\text{for any } e^{(i)}\in (\widetilde{E})^{1}_{Z},\\
\widetilde{s}(e) &= s(e),\qquad
\widetilde{r}(e) = r(e),\qquad
\widetilde{w}(e) = w(e)
\quad\text{for any } e\in (\widetilde{E})^{1}_{Z^{c}}.
\end{aligned}
\]
We call $(\widetilde{E},\widetilde{w})$ the 
\emph{$\mathcal{P}1$-weighted graph} of $(E,w)$.

\end{definition}

We shall use the following structural properties of the $\mathcal{P}1$-weighted
graph \((\widetilde{E},\widetilde{w})\), which we call a \emph{sink-separated weighted graph}.

\begin{definition} \label{def:sink-separatedweightedgraph} We say that a weighted graph $(E,w)$ is \emph{sink-separated}  if it satisfies the following:

\begin{enumerate}
    \item [(1)] Every weighted edge in \((E,w)\) has range a sink. 
\item [(2)] No two distinct weighted edges have the same source or the same range.
\end{enumerate} In this case, for every
\(v\in r(E_w)\), we denote the unique weighted edge with range $v$ by \(g^v\) .
\end{definition} 
\begin{remark}
    
Equivalently, $(E,w)$ is sink-separated if the following are satisfied: 
\begin{enumerate}
\item $r(E_w)\subseteq \textnormal{Sink}(E)$; 
    \item $r|_{E_w}: E_w\longrightarrow \textnormal{Sink}(E)$ is injective; and  
    \item $s|_{E_w}: E_w\longrightarrow E^0$ is injective.
\end{enumerate}

\end{remark}
The following lemma directly follows from  \cite[Lemma 9]{Preusser2019WeightedLP}. 
\begin{lemma}
    Let $(E,w)$ be an (LPA) weighted graph and $(\widetilde{E},\widetilde{w})$ its associated $\mathcal{P}1$-weighted graph. Then $(\widetilde{E},\widetilde{w})$ is sink-separated. 
\end{lemma}

\begin{definition}[Step $\mathcal{P}2$] Suppose $(E,w)$ is a sink-separated weighted graph. Construct the unweighted graph $F=(F^0,F^1, s_F, r_F)$ as follows: 
For $v\in r(E_w)$,
\begin{enumerate}[leftmargin=1.8cm]
    \item [($\mathcal{P}2.1$)]  replace $v$ with $w(g^v)$ vertices $v^{(1)}, v^{(2)}, \dots, v^{(w(g^v))}$;
    \item [($\mathcal{P}2.2$)] replace $g^v$ with  $w(g^v)$ unweighted edges $(g^v)^{(1)}, (g^v)^{(2)}, \dots, (g^v)^{(w(g^v))}$ such that $s_F((g^v)^{(1)})=s(g^v)$ and $r_F((g^v)^{(1)})=v^{(1)}$,  for each $i\geq 2$, and $s_F((g^v)^{(i)})=v^{(i)}$ and $r_F((g^v)^{(i)})=s(g^v)$;
    \item [($\mathcal{P}2.3$)]  replace unweighted edge $e$ such that $r(e)=v$ with $w(g^v)$ unweighted edges $e^{(1)}, e^{(2)}, \dots, e^{(w(g^v))}$ such that $s_F(e^{(i)})=s(e)$ and $r_F(e^{(i)})=v^{(i)}$.  
\end{enumerate}

More formally, for sink-separated weighted graph $(E,w)$ where $E=(E^0, E^1,s,r)$, $F$ is the unweighted graph defined as follows:\\
$F^{0} = M \sqcup N$ where 
\[
M=E^0\setminus r(E_w)\qquad \textnormal{and} \qquad N=\{v^{(1)}, v^{(2)}, \dots, v^{(w(g^v)}\}~|~v\in r(E_w)\},
\]
$F^1=A\sqcup B \sqcup C\sqcup D$ where
\[
A=\{e\mid e\in E_{uw},\ r(e)\notin r(E_{w})\},
\]
\[
B=\{e^{(1)},\ldots,e^{(w(g^{r(e)}))}
   \mid e\in E_{uw},\ r(e)\in r(E_{w})\},
\]
\[
C=\{e^{(1)}\mid e\in E_{w}\},
\]
\[
D=\{e^{(2)},\ldots,e^{(w(e))}\mid e\in E_{w}\},
\]
and the maps $s_F$ and $r_F$ given by
\[
\begin{aligned}
s_F(e)&=s(e), & r_F(e)&=r(e) && (e\in A),\\
s_F(e^{(i)})&=s(e), & r_F(e^{(i)})&=r(e)^{(i)}
&& (e^{(i)}\in B),\\
s_F(e^{(1)})&=s(e), & r_F(e^{(1)})&=r(e)^{(1)}
&& (e^{(1)}\in C),\\
s_F(e^{(i)})&=r(e)^{(i)}, & r_F(e^{(i)})&=s(e)
&& (e^{(i)}\in D).
\end{aligned}
\]
We call $F$ the \emph{$\mathcal{P}2$-graph} of $(E,w)$. For $v\in r(E_w)$ we call $v^{(i)}\in F^0$ a \emph{split vertex} and for $w\not \in r(E_w)$, the corresponding vertex $w\in F^0$ is called an \emph{unsplit vertex}. Moreover, we call $e^{(i)}\in B\cup C\cup D$ a \emph{split edge} and $e\in A$ an \emph{unsplit edge}. 

\end{definition}
If the context is clear, we sometimes write $s_F$ and $r_F$ simply as $s$ and $r$, respectively.

\begin{remark}\label{rem:StepP2-v1-sink}
    Let $(E,w)$ be a sink-separated weighted graph and $F$ its associated $\mathcal{P}2$-graph. Then, by the conditions for a sink-separated graph, we have that for every $v\in r(E_w)$, $v$ is a sink. Notice that by Steps $(\mathcal{P}1.2)$ and $(\mathcal{P}1.3)$, it follows that $v^{(1)}$ is a sink in $F$.  
\end{remark}

\begin{lemma}[{\cite[Lemmas 9 and 11]{Preusser2019WeightedLP}}]\label{lem:P1P2isomorphic}
 Let $\K$ be a field, $(E,w)$ be an (LPA) weighted graph, $(\widetilde{E},\widetilde{w})$ be its $\mathcal{P}1$-weighted graph and $F$ be the $\mathcal{P}2$-graph of $(\widetilde{E},\widetilde{w})$. Then as $\K$-algebras, 
 \[L_\K(E,w)\cong L_\K(\widetilde{E},\widetilde{w})\cong L_\K(F).\]    
\end{lemma}
The following is a direct consequence of Lemma \ref{lem:P1P2isomorphic}.

\begin{proposition}
    Let $\K$ be a field, $(E,w)$ be an (LPA) weighted graph, $(\widetilde{E},\widetilde{w})$ be its $\mathcal{P}1$-weighted graph and $F$ be the $\mathcal{P}2$-graph of $(\widetilde{E},\widetilde{w})$. Then as monoids,
    \[\mathcal{V}(L_\K(E,w))\cong \mathcal{V}(L_\K(\widetilde{E},\widetilde{w}))\cong \mathcal{V}(L_\K(F)).\]    
Hence, 
    \[\mathcal{M}(E,w)\cong \mathcal{M}(\widetilde{E},\widetilde{w})\cong M_F.\]
    
\end{proposition}

\begin{remark}\label{rem:cancellativitypreserved}
A commutative monoid $(M, +)$ is said to be \emph{cancellative} if $a+b=c+b$ implies $a=c$ which is preserved under monoid isomorphisms. Hence, $\mathcal{M}(E,w)$ is cancellative if and only if $M_F$ is cancellative.
\end{remark}

\subsection{Cancellativity-Detecting Auxiliary Graph}
\label{subsec:Cancellativity}

\begin{definition}\label{def:auxiliarygraph} For a weighted graph $(E,w)$, the \emph{auxiliary (unweighted and directed) graph} of $(E,w)$ is $E_{\rm{aux}}=(E_{\rm{aux}}^0, E_{\rm{aux}}^1, s_{\rm{aux}}, r_{\rm{aux}})$, where
\[
E_{\rm{aux}}^0=E^0,\qquad 
E_{\rm{aux}}^1
=
E_{uw}^1\sqcup \{\overline g\mid g\in E_w^1\},   
\]
and the maps $s_{\rm{aux}}, r_{\rm{aux}}: E_{\rm{aux}}^1\longrightarrow E_{\rm{aux}}^0$ defined by 
\[
s_{\rm{aux}}(e)=s(e),\qquad s_{\rm{aux}}(\overline{f})=r(f),  \qquad r_{\rm{aux}}(e)=r(e), \qquad \textnormal{and}\qquad  r_{\rm{aux}}(\overline{f})=s(f),\]
for every $e\in E^1_{uw}$ and $f\in E^1_w$. We call $\overline{f}\in E_{\rm{aux}}^1$ the \emph{bar edge} associated with $f\in E^1_w$.
\end{definition}

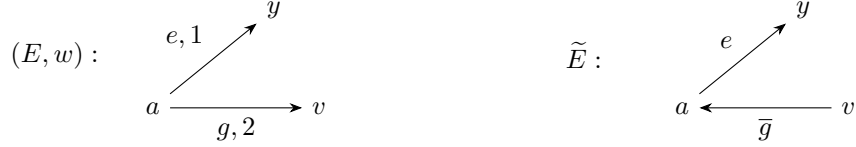
\begin{figure}[H]
    \centering
    \begin{tikzpicture}[
    >={Stealth},
    every node/.style={font=\small}
]
    \begin{scope}
        \node (a0) at (-1.3,0.7) {$(E,w):$};
        
        \node (a1) at (0,0) {$a$};
        \node (y1) at (1.6,1.3) {$y$};
        \node (v1) at (2.2,0) {$v$};

        \draw[->] (a1) -- node[above left] {$e,1$} (y1);
        \draw[->] (a1) -- node[below] {$g,2$} (v1);
    \end{scope}

    \begin{scope}[xshift=7cm]
        \node (a0) at (-1.3,0.7) {$E_{\rm{aux}}:$};
        
        \node (a2) at (0,0) {$a$};
        \node (y2) at (1.6,1.3) {$y$};
        \node (v2) at (2.2,0) {$v$};

        \draw[->] (a2) -- node[above left] {$e$} (y2);
        \draw[->] (v2) -- node[below] {$\overline{g}$} (a2);
    \end{scope}
\end{tikzpicture}
    \caption{Weighted graph $(E,w)$ and its associated auxiliary directed graph $E_{\rm{aux}}$}
\label{fig:auxiliarygraph}
\end{figure}

\begin{proposition}\label{prop:cycle-exit}\label{prop:characterization_CE1_CE2}
Let $(E,w)$ be a sink-separated weighted graph, and let $E_{\rm{aux}}$ and $F$ be its associated auxiliary graph and $\mathcal{P}2$-graph, respectively. Then $F$ has no cycle with an exit if and only if the following two
conditions hold:
\begin{enumerate}
    \item[\emph{(CE1)}] \(E\) has no cycle with an exit.
    \item[\emph{(CE2)}] No cycle in \(E_{\rm{aux}}\) contains a bar edge.
\end{enumerate}
\end{proposition}

\begin{proof}
    $(\Rightarrow)$ Suppose $F$ has no cycle with an exit. We shall first prove (CE1). To the contrary, suppose $E$ has a cycle $C$ with an exit $f$. Since every vertex in $r(E_w)$ is a sink in $E$, no vertex in $C$ lies in $r(E_w)$. Hence $C$ is also a cycle in $F$.

    \noindent Case 1: $f\in E_{uw}$ and $r(f)\not \in r(E_w)$. Then $f$ is unchanged in $F$ and hence an exit for $C$ in $F$.\\
    \noindent Case 2: $f\in E_{uw}$ and $r(f) \in r(E_w)$. Then by Step $(\mathcal{P}2.3)$, $f$ is replaced with edges $f^{(i)}$ with $s(f^{(i)})=s(f)$ in $F$. Hence, each $f^{(i)}$ is an exit for $C$ in $F$.\\
    Case 3: $f\in E_w$. Then by Step $(\mathcal{P}2.2)$, $f=g^v$ for some $v\in r(E_w)$ and there exists an edge $f^{(1)}\in F^1$ with $s(f^{(1)})=s(f)$.

In all cases, we obtain a contradiction. Thus, (CE1) holds.

    Now, we show (CE2). Suppose in contrary that $E_{\rm{aux}}$ contains a cycle $C$ with a bar edge $\overline{g}$, $g\in E_w$. Permute $C$ to obtain a cycle $C'=\overline{g}p$  where $p\in \textnormal{Path}(E_{\rm{aux}})$ with $s(p)=r(\overline{g}):=a$ and $r(p)=s(\overline{g}):=v$. Hence, in $E$, $s(g)=a$ and $r(g)=v$. Moreover, $g=g^v$ and by Step $(\mathcal{P}2.2)$, there exists a vertex $v^{(2)}\in F^0$ and an edge $g^{(2)}\in F^1$ with $s_F(g^{(2)})=v^{(2)}$ and $r_F(g^{(2)})=s(g)=a$. 
    
    Now the path $p\in \textnormal{Path}(E_{\rm{aux}})$ gives rise to a path in $F$. We associate to each vertex $v \in E_{\rm{aux}}^0$ the vertex $\hat v\in F^0$: $$\hat v = \begin{cases}
        v & \text{if~}v \notin r(E_w), \\
        v^{(2)} & \text{otherwise},
    \end{cases}$$ and to each edge $e\in E_{\rm{aux}}^1$, the edge $\hat e\in F^1$: $$\hat e = \begin{cases}
        e & \text{if~}e \in E_{uw} \text{~and~} r(e) \notin r(E_w),\\
        e^{(2)} & \text{if~}e \in E_{uw}\text{~and~} r(e)\in r(E_w),\\
        h^{(2)} & \text{if~}e = \overline{h}, \ h \in E_w.
    \end{cases}$$ By steps $(\mathcal{P}2.2)$ and $(\mathcal{P}2.3)$, these correspondences preserve sources and range in $F$. Hence $p \in \textnormal{Path}(E_{\rm{aux}})$ determines a path $p' \in \textnormal{Path}(F)$ from $s(g)$ to $r(g)^{(2)}$. Since $g^{(2)}:r(g^{(2)})\longrightarrow s(g)$ the path $g^{(2)}p'$ is a cycle in $F$ with an exit $g^{(1)}$, a contradiction. Hence, (CE2)  holds.

\noindent $(\Leftarrow)$ Suppose (CE1) and (CE2) hold. Suppose $F$ contains a cycle $C$ with an exit. We first show that no edge in $C$ arises from a weighted edge in $E$. Let $g\in E_w$, $s(g):=a$, $r(g)=v$ and $w(g)=n$. Since $(E,w)$ is sink-separated, $v$ is a sink. By Remark \ref{rem:StepP2-v1-sink}, $g^{(1)}$ is a sink in $F$, hence, $g^{(1)}$ is not an edge in $C$. Suppose $C$ contains some edge $g^{(i)}$, $2\leq i\leq n$. Then we obtain a cycle $C'$ in $E_{\rm{aux}}$ containing $\overline{g}$, since edges $e\in E_{uw}$ are edges retained in $E_{\rm{aux}}^1$, which will then correspond to edges with the appropriate respective source and range in $F$; and edges $h\in E_w$  will correspond to the edges $h^{(i)}\in F^1$ ($i>1$) and $\overline{h}\in E_{\rm{aux}}^1$ which are both reversed in orientation. This contradicts (CE2). Hence no edge in $C$ arise from a weighted edge in $E$. Moreover, it follows that $C$ contains no split vertex $v^{(i)}$. Thus, $C$ is a cycle in $E$. By (CE1), $C$ has no exit in $E$. We show that this contradicts the assumption that $C$ has an exit in $F$.

Let $f\in F^1$ be an exit for $C$ in $F$. Then $s(f)$ is an unsplit vertex. 

\noindent Case 1: $f$ is an unsplit edge. Then $f$ is also an exit for $C$ in $E$.

\noindent Case 2: $f=e^{(1)}$ for some $e\in E_w$, which is not an edge in $C$ in $F$. Then $s(f)=s(e)$ and $e$ is an exit for $C$ in $E$.

\noindent Case 3: $f=e^{(j)}$ for some $e\in E_{uw}$ with $r(e)\in r(E_w)$. Then $e$ is an exit for $C$ in $E$.

In all cases, we obtain a contradiction to (CE1). Thus, $F$ has no cycle with an exit.\end{proof}

The following result is similarly proved as Proposition \ref{prop:characterization_CE1_CE2} and we leave the proof to the reader.

\begin{corollary}
    
\label{cor:cycle-exit-2}
Let $(E,w)$ be a sink-separated weighted graph, and let $E_{\rm{aux}}$ and $F$ be its associated auxiliary graph and $\mathcal{P}2$-graph, respectively. Then $F$ is acyclic if and only if the two conditions hold:
\begin{enumerate}
    \item [\emph{(AC)}] $E$ is acyclic.
    \item [\emph{(CE2)}] No cycle in \(E_{\rm{aux}}\) contains a bar edge.
\end{enumerate}\end{corollary}

\begin{lemma}[{\cite[Lemma 5.5]{AraHazratLiSims2018}}] \label{lem:cancellativeM_E} Let $E$ be an arbitrary graph. The monoid $M_E$ is cancellative if and only if no cycle in $E$ has an exit. In particular, if $E$ is acyclic, then $M_E$ is cancellative. 
\end{lemma}

Now, we are ready to characterize cancellativity of weighted graph monoids of (LPA) graphs.
\begin{theorem}\label{theo:M(E,w)CancellativeCharacterization}
        Let $(E,w)$ be an (LPA) weighted graph, $(\widetilde{E},\widetilde{w})$ its $\mathcal{P}1$-weighted graph, and $\widetilde{E}_{\rm{aux}}$ the associated auxiliary graph of $(\widetilde{E},\widetilde{w})$. Then $\mathcal{M}(E,w)$ is cancellative if and only if 
    \begin{enumerate}
    \item[\emph{(CE1)}] \(\widetilde{E}\) has no cycle with an exit.
    \item[\emph{(CE2)}] No cycle in \(\widetilde{E}_{\rm{aux}}\) contains a bar edge.
\end{enumerate}
\end{theorem}

\begin{proof}
    Directly follows from Proposition \ref{prop:characterization_CE1_CE2}, Lemma \ref{lem:cancellativeM_E} and Remark \ref{rem:cancellativitypreserved}. 
\end{proof}

\begin{corollary}\label{cor:acyclic(LPA)_Cancellative}
    Let $(E,w)$ be an acyclic (LPA) weighted graph, $(\widetilde{E},\widetilde{w})$ its $\mathcal{P}1$-weighted graph, and $\widetilde{E}_{\rm{aux}}$ the associated auxiliary graph of $(\widetilde{E},\widetilde{w})$. Then $\mathcal{M}(E,w)$ is cancellative if and only if no cycle in $\widetilde{E}_{aux}$ contains a bar edge.
\end{corollary}

\begin{example}
   Consider the weighted graph $(E,w)$ and its auxiliary graph $E_{\rm{aux}}$ in Figure \ref{fig:auxiliarygraph}.
As $Z:=T\bigl(r(E_w)\bigr)=\{v\}$ consists of a sink $v$, we have
$s^{-1}(Z)=\varnothing$ and hence we have the $\mathcal{P}1$-weighted graph $(\widetilde{E},\widetilde{w})=(E,w)$, with
$g_v=g$. Moreover, we have the auxiliary graph $E_{\rm{aux}}=\widetilde{E}_{\rm{aux}}$. The weighted graph $(\widetilde{E},\widetilde{w})$ is acyclic and $E_{\rm{aux}}=\widetilde{E}_{\rm{aux}}$ is also acyclic. By Corollary \ref{cor:acyclic(LPA)_Cancellative}, $\mathcal{M}(E,w)$ is cancellative. 

More precisely, the $\mathcal{P}2$-graph of $(\widetilde{E},\widetilde{w})$ is: 
\begin{center}
    \begin{tikzpicture}[>=stealth]
    \node (F) at (-1.6, -0.3) {$F:$};
    \node (v2) at (0,0) {$v^{(2)}$};
    \node (a) at (2.2,0) {$a$};
    \node (v1) at (4.4,0) {$v^{(1)}$};
    \node (y) at (2.2,-1.4) {$y$};

    \draw[->] (v2) -- node[above] {$g^{(2)}$} (a);
    \draw[->] (a) -- node[above] {$g^{(1)}$} (v1);
    \draw[->] (a) -- node[right] {$e$} (y);
\end{tikzpicture}
\end{center}
Hence
$$
M_F=\bigl\langle a,y,v^{(1)},v^{(2)} \ \bigm|\
a=y+v^{(1)},\ v^{(2)}=a \bigr\rangle;
$$
see Example \ref{ex:weighted_graph_monoid}. Thus, $
M_F\cong\mathbb N^2$, 
which is cancellative.

Calculating $\mathcal{M}(E,w)$, since $k_a=2$, $\mathcal{M}(E,w)$ is generated by the set $\{a,y,v,q_1^a\}$ subject to the relations
$$
(\mathfrak{R}1(a))~~q^a_0+\bigl(w_1(a)-w_0(a)\bigr)a=q^a_1+y
\qquad \text{and}\qquad (\mathfrak{R}2(a))~~
q^a_1+\bigl(w_2(a)-w_1(a)\bigr)a=q^a_2+v ,
$$
with $q^a_0=q^a_2=0$, hence reduced to 
$$
(\mathfrak{R}1(a))~~a=q^a_1+y\qquad \text{and}\qquad (\mathfrak{R}2(a))~~ q^a_1+a=v .
$$
Accordingly, 
$$
\mathcal{M}_{(E,w)}=\bigl\langle a,y,v,q^a_1 \ \bigm|\
a=q^a_1+y,\ q^a_1+a=v \bigr\rangle \cong \mathbb{N}^2\cong M_F.
$$

\end{example}

\begin{example}\label{ex:Cancellative_example_2}
    Consider the weighted graphs $(E,w_1)$ and $(E,w_2)$:

    \begin{center}
      \begin{tikzpicture}[>=stealth, scale=1, every node/.style={scale=1}]
    \node (u1) at (0,0) {$u$};
    \node (v1) at (2.2,0) {$v$};
    \node at (-1.2,0) {$(E,w_1):$};

    \draw[->, shorten <=1pt, shorten >=1pt]
        (u1) to[bend left=35]
        node[above] {$e,2$} (v1);
    \draw[->, shorten <=1pt, shorten >=1pt]
        (u1) to[bend right=35]
        node[below] {$f,1$} (v1);

    \node (u2) at (6.8,0) {$u$};
    \node (v2) at (9,0) {$v$};
    \node at (5.6,0) {$(E,w_2):$};

    \draw[->, shorten <=1pt, shorten >=1pt]
        (u2) to[bend left=35]
        node[above] {$e,2$} (v2);
    \draw[->, shorten <=1pt, shorten >=1pt]
        (u2) to[bend right=35]
        node[below] {$f,2$} (v2);

\end{tikzpicture}
\end{center}

Notice that $(E,w_1)$ is an (LPA) graph and $(E,w_2)$ is not. By Theorem \ref{theo:WLPA=LPA_iff_LPAGraph}, it follows that $L_\K(E,w_1)\not \cong L_\K(E,w_2)$. Moreover, we also have $(\widetilde{E},\widetilde{w_1})=(E,w_1)$ and the auxiliary graph $\widetilde{E}_{\rm{aux}}=E_{\rm{aux}}$ contains a cycle $C=\overline{e}f$. By Theorem~\ref{theo:M(E,w)CancellativeCharacterization}, it follows that $\mathcal{M}(E,w_1)$ is not cancellative while by \cite[Theorem 2.6]{dhn2026}, $\mathcal{M}(E,w_2)$ is cancellative. 
\end{example}

\medskip

\subsection{Graded Extension of Preusser's Two-Step Construction} \label{subsec:Vgr-onoid}

For a weighted graph $(E,w)$,  denote by $\mathcal V^{\mathrm{gr}}(L_\K(E,w))$ the commutative monoid of isomorphism classes of finitely generated graded projective unital right $L_\K(E,w)$-modules. 

The preceding results show that, within Condition~(LPA), the study of the $\mathcal V$-monoid of a weighted Leavitt path algebra can be transferred to an ordinary Leavitt path algebra through Preusser's two-step construction \cite{Preusser2019WeightedLP}. It is natural to ask whether this construction can be utilized in graded setting, that is, for the case of the weighted Leavitt path algebra's $\mathcal{V}^{\operatorname{gr}}$-monoid.

We briefly consider this question for the \emph{standard grading} for weighted Leavitt path algebras. For a weighted graph $(E,w)$, let 
\[
\lambda(E,w)=\sup\{w(v):v\in E^0\}
\]
when the supremum is finite, and set $\lambda(E,w)=\omega$ otherwise, where $\omega$ is the smallest infinite ordinal. Let $(E,w)$ be a weighted graph with $\lambda(E,w)<\infty$. Then  $L_\K(E,w)$ admits the \emph{standard $\mathbb Z^{\lambda(E,w)}$-grading} given by  $$ \deg(v)=0,\qquad \deg(e_i)=\epsilon_i,\qquad \deg(e_i^*)=-\epsilon_i, $$ 
for every $v\in E^0$, $e\in E^1$, and $1\leq i\leq w(e)$, where $\epsilon_i$ denotes the $i$-th standard basis vector of $\mathbb Z^{\lambda(E,w)}$. Equivalently, this is the grading induced by the admissible weight map $W_s$ given by $W_s(e_i)=\epsilon_i$; see \cite{preusser2021weighted} for more details.

For a vertex-weighted graph $(E,w)$, Damalerio, Hazrat and Nam in \cite{dhn2026} defined the so-called \emph{talented monoid} $T_{(E,w)}$ and showed that this monoid retrieves the $\mathcal V^{\mathrm{gr}}(L_\K(E,w))$ with respect to the \emph{natural $\mathbb Z$-grading} given by $$ \deg(v)=0,\qquad \deg(e_i)=1,\qquad \deg(e_i^*)=-1, $$ for every $v\in E^0$, $e\in E^1$, and $1\leq i\leq w(e)$. For arbitrary weighted graphs, however, a corresponding description of $\mathcal V^{\mathrm{gr}}(L_\K(E,w))$ with respect to the standard $\mathbb Z^{\lambda(E,w)}$-grading has not been studied extensively in the existing literature, and we do not address this problem in full here. Instead, taking advantage of the reduction available under Condition~(LPA), we show that Preusser's construction is compatible with the standard grading after suitable degrees are assigned to the edges of the resulting unweighted graph. Hence, we have the following lifting of \cite[Theorem 1]{Preusser2019WeightedLP} in the graded setting.

\begin{theorem} \label{theo:gradedisomorphism_of_LPA_graphs}
Let $(E,w)$ be row-finite (LPA) weighted graph and $\lambda(E,w)<\infty $. Then there exists row-finite unweighted graph $F$ and a $\mathbb{Z}^{\lambda(E,w)}$-grading on $L_\K(F)$ such that $ L_\K(E,w)\cong_{\mathrm{gr}}L_\K(F) $ as $\mathbb Z^{\lambda(E,w)}$-graded $\K$-algebras. \end{theorem}

\begin{proof} 
Let $(E,w)$ be an {\rm (LPA)} weighted graph with $\lambda(E,w)<\infty$, and let $$ (E,w)\xlongrightarrow{~\mathcal{P}1~}(\widetilde{E},\widetilde{w})\xlongrightarrow{~\mathcal{P}2~} F $$ be Preusser's two-step construction described in Subsection~\ref{subsec:2-Step_(LPA)Graphs}. 
Put $\lambda=\lambda(E,w)$.
We define a map $ d:F^1\longrightarrow\mathbb Z^{\lambda} $ such that $ L_\K(E,w)\cong_{\mathrm{gr}}L_\K(F) $ as $\mathbb Z^{\lambda}$-graded $\K$-algebras, where $L_\K(E,w)$ is equipped with its standard grading and $L_\K(F)$ is graded defined by $ \deg(f)=d(f)$, $\deg(f^*)=-d(f)$ for every $f\in F^1$.

For Step~$\mathcal{P}1$, let $Z=T(r(E_w))$. Let $e\in E^1$. If $s(e)\in Z$, the edge $e$ is replaced by the unweighted edges $e^{(1)},\ldots,e^{(w(e))}$ with reversed orientation. Assign $\deg(e^{(i)}_1)=-\epsilon_i $. If $s(e)\notin Z$, retain $\deg(e_i)=\epsilon_i$. Moreover, set $\deg(v)=0$ for all $v\in E^0$ and $\deg(f^*)=-\deg(f)$ for all $f\in (\widetilde{E})^1$. 
By \cite[Lemma 9]{Preusser2019WeightedLP}, the Step~$\mathcal P1$  isomorphism $ \phi_1:L_\K(E,w)\longrightarrow L_\K(\widetilde{E},\widetilde{w}) $ satisfies 
$$ \phi_1(e_i)= 
\begin{cases} 
(e^{(i)}_1)^*&\text{if~}s(e)\in Z,\\ 
e_i&\text{otherwise}. 
\end{cases} $$ Thus, if $s(e)\in Z$, $ \deg\big((e^{(i)}_1)^*\big)=\epsilon_i=\deg(e_i)$, and otherwise the degree is unchanged. Hence $\phi_1$ is a $\mathbb{Z}^\lambda$-graded isomorphism.

For Step~$\mathcal P2$, we define the grading on $F$ from the construction of its edges. If an unweighted edge $e$ remains unchanged, set $d(e)=\deg(e_i)$. If $e$ is replaced by $e^{(1)},\ldots,e^{(m)}$, set $ d(e^{(j)})=\deg(e_1) $ for every $j$. Finally, if $e$ is weighted, set $ d(e^{(1)})=\deg(e_1)$ and $d(e^{(i)})=-\deg(e_i)$ for each $i> 1$. Under Step~$\mathcal P2$ the isomorphism $$ \phi_2:L_\K(\widetilde{E},\widetilde{w})\longrightarrow L_\K(F) $$ in \cite[Lemma 11]{Preusser2019WeightedLP}, an edge generator $e_i$ is sent either to an edge, to a sum of edges of the same degree, or, for $i>1$ on a weighted edge, to $(e^{(i)})^*$. By the definition of $d$, in each case $$ \deg(\phi_2(e_i))=\deg(e_i). $$ 
The corresponding statement holds for ghost edges, while vertices are sent to homogeneous elements of degree zero. Therefore $\phi_2$ is a  $\mathbb{Z}^{\lambda(E,w)}$-graded isomorphism. Consequently, $\phi_2\phi_1$ gives the required $\mathbb{Z}^{\lambda(E,w)}$-graded isomorphism $ L_\K(E,w)\cong_{\mathrm{gr}}L_\K(F). $\end{proof}

\begin{remark} 
The grading on $L_\K(F)$ in the preceding theorem is, in general, not the usual $\mathbb Z$-grading of an unweighted Leavitt path algebra. The degree of an edge of $F$ is determined by its origin in Preusser's construction. 
\end{remark}

For a group $\Gamma$ and a commutative monoid $M$, we say $M$ is a \emph{$\Gamma$-monoid} if $\Gamma$ acts on $M$ by monoid automorphisms. Hence, under the $\mathbb{Z}^{\lambda(E,w)}$-grading, $\mathcal{V}^{\operatorname{gr}}(L_\K(E,w))$ and $\mathcal{V}^{\operatorname{gr}}(L_\K(F))$ become  $\mathbb{Z}^{\lambda(E,w)}$-monoids and we have the following direct consequence.

\begin{corollary} \label{cor:Vgr(L(E,w))=Vgr(L(F))}
Let $(E,w)$ be a row-finite (LPA) weighted graph with  $\lambda(E,w)<\infty$, and let $F$ be the row-finite unweighted graph as in Theorem \ref{theo:gradedisomorphism_of_LPA_graphs} then as $\mathbb Z^{\lambda(E,w)}$-monoids, $ \mathcal V^{\mathrm{gr}}(L_\K(E,w)) \cong \mathcal V^{\mathrm{gr}}(L_\K(F))$.   
\end{corollary}

Hazrat \cite{Hazrat:2013_GradedGrothendieckGroupClassificationLeavittPathAlgebras} posed the \emph{Graded Classification Conjecture} for Leavitt path algebras which claims that for the (unweighted) Leavitt path algebras, the \emph{graded Grothendieck group} (which is obtained by completing the $\mathcal{V}^{\operatorname{gr}}$-monoid to a group) under the standard grading along with its ordering and its module structure is a complete invariant and which was also  proved  therein for the class of \emph{polycephaly graphs}; see \cite{Hazrat:2013_GradedGrothendieckGroupClassificationLeavittPathAlgebras,hazrat2020talented, Bock_Hazrat_Sebandal2025} for some of the works on the conjecture. In \cite{dhn2026}, the conjecture's extension to weighted Leavitt path algebras of finite vertex-weighted graphs under the natural $\mathbb{Z}$-grading was investigated. In light of Corollary \ref{cor:Vgr(L(E,w))=Vgr(L(F))}, can tools and techniques in the unweighted case be applied to (LPA) weighted case under the standard $\mathbb{Z}^{\lambda(E,w)}$-grading?

\section*{Acknowledgments}

The first author gratefully acknowledges Dr.~Ardeline Mary Buhphang, their Ph.D. supervisor, for her guidance and support during the preparation of this work. The second author acknowledges Dr.~Christopher Bernido of the Research Center for Theoretical Physics (Central Visayan Institute Foundation) and Dr.~H\r{a}kan Sollervall of Linnaeus University for hosting her during the research fellowships at their respective institutions, as well as Dr.~Berlita Disca of Mindanao State University -- General Santos City, Dr.~Vyachelsav Futorny of Southern University of Science and Technology, and Dr.~Ardeline Mary Buhphang of North-Eastern Hill University for the research stays which allowed the completion of this paper. Both authors would also like to thank Mr.~Remarl Joseph Damalerio of MSU -- General Santos City for the insightful discussions with the second author, which led to the paper's  improvement.

\printbibliography

\end{document}